\documentclass{article}
\usepackage[top=3.5cm,bottom=3.5cm,left=3cm,right=3cm,a4paper]{geometry}
\usepackage{amsfonts,amsmath,amssymb,amsthm,upgreek}
\usepackage{mathtools}
\usepackage{float,fancyhdr,siunitx}
\usepackage[dvipsnames]{xcolor}
\usepackage{setspace}
\usepackage{comment}
\usepackage{stmaryrd}
\usepackage{subcaption}

\usepackage{lineno}
\usepackage{empheq,mdframed}
\usepackage[dvipsnames]{xcolor}
\usepackage[normalem]{ulem} %% to cross text with \sout
\usepackage[colorlinks]{hyperref}
\hypersetup{
    colorlinks=true,
    linkcolor={blue!50!black},
    citecolor={red!80!black},
    urlcolor= {blue!50!black}
}
\usepackage{enumitem,mdframed}
\usepackage{empheq,longtable}
\usepackage[]{algorithm,algorithmic}
\usepackage{tikz,pgfplots}
\usepgfplotslibrary{groupplots}
\usetikzlibrary{positioning,arrows,fit,patterns}
\pgfplotsset{compat=newest}

\newtheorem{remark}{Remark}
\newtheorem{assumption}{Assumption}
\newtheorem{notation}{Notation}
\newtheorem{definition}{Definition}
\newtheorem{proposition}{Proposition}

\usepackage[capitalize,nameinlink]{cleveref}
\crefname{section}   {Section}   {Sections}
\crefname{subsection}{Subsection}{Subsections}
\Crefname{section}   {Section}   {Sections}
\Crefname{subsection}{Subsection}{Subsections}
\Crefname{figure}    {Figure}    {Figures}

\crefdefaultlabelformat{#2\textup{#1}#3}
\crefname{equation}{}{}
\Crefname{equation}{}{}
\crefname{proposition}{Proposition}{Propositions}
\Crefname{proposition}{Proposition}{Propositions}
\crefname{definition} {Definition} {Definitions}
\Crefname{definition} {Definition} {Definitions}
\crefname{theorem}    {Theorem}    {Theorems}
\Crefname{theorem}    {Theorem}    {Theorems}
\crefname{remark}     {Remark}     {Remarks}
\Crefname{remark}     {Remark}     {Remarks}
\crefname{assumption} {Assumption} {Assumptions}
\Crefname{assumption} {Assumption} {Assumptions}
\crefname{claim} {Claim} {Claims}
\Crefname{claim} {Claim} {Claims}

\makeatletter
\renewcommand*\env@matrix[1][\arraystretch]{%
  \edef\arraystretch{#1}%
  \hskip -\arraycolsep
  \let\@ifnextchar\new@ifnextchar
  \array{*\c@MaxMatrixCols c}}
\makeatother
\newcommand*\samethanks[1][\value{footnote}]{\footnotemark[#1]}

\newcommand{\ii}{\mathrm{i}}
\newcommand{\bx}{\mathbf{x}}
\newcommand{\by}{\mathbf{y}}
\newcommand{\rmd}{\mathrm{d}}

\newcommand{\bv}{{\mathbf{v}}}

\newcommand{\bn}{\mathbf{n}}
\newcommand{\bd}{\mathbf{d}}

\newcommand{\bm}{\mathbf{m}}

\newcommand{\Cov}{\mathbf{Cov}}

\newcommand{\srcs}{\Sigma_\text{src}}
\newcommand{\rcvs}{\Sigma_\text{rcv}}
\newcommand{\brmU}{\boldsymbol{\mathrm{U}}}
\newcommand{\brmG}{\boldsymbol{\mathrm{G}}}
\newcommand{\brmC}{\boldsymbol{\mathrm{C}}}
\newcommand{\brmS}{\boldsymbol{\mathrm{S}}}

\newcommand{\dof}{\mathrm{dof}}
\newcommand{\ndof}{n^\dof}
\newcommand{\comp}{\mathrm{comp}}

\newcommand{\brmV}{\boldsymbol{\mathrm{V}}}

\newcommand{\ncomp}{n_{\mathrm{comp}}}
\renewcommand{\dim}{\mathrm{d}}

\title{Full cross-correlation inversion for quantitative passive imaging with time-harmonic acoustic waves}

\author{Jean Dutheil\thanks{Project-Team Makutu, Inria Bordeaux, University of Pau and Pays de l’Adour, Pau, France.} \, and Florian Faucher\samethanks[1]}

\date{\today}

\begin{document}

\maketitle

\numberwithin{equation}{section}

% --------------------------------------
\begin{abstract} 

We consider the inverse problem for the quantitative reconstruction of 
physical properties in the context of passive imaging, where ambient 
wavefields are used to infer a medium. The data are 
modeled as a superposition of waves generated by stochastic sources. 
In this work, we focus on time-harmonic acoustic wave propagation and 
assume that the stochastic sources exciting the medium are zero-mean 
and spatially uncorrelated. Under these assumptions, the expected value 
of the cross-correlation between signals recorded at two locations can 
be related to the deterministic Green's function and the covariance of 
the source terms. We follow a first-order formulation of the wave equation, 
which enables the treatment of correlations between different types of 
wavefields. A numerical framework is developed for the resulting nonlinear 
inverse problem. The quantitative reconstruction is carried out using an 
iterative minimization scheme, in which the gradient of the misfit functional 
is computed via the adjoint-state method. Numerical experiments in 
two and three dimensions are performed using synthetic data, and inversions 
based on the expected value of cross-correlations are compared with those 
relying on direct wavefield measurements from active-source acquisitions.

\end{abstract}
\tableofcontents
% --------------------------------------------------------------------------
\section{Introduction}
\label{section:intro}
% --------------------------------------------------------------------------

In passive imaging, background ambient signals are exploited to probe a 
medium of interest. 
Passive imaging is in particular used in asteroseismology and planetary 
seismology, where continuous vibrations of the bodies are recorded, 
for example in helioseismology \cite{Gizon2005} and terrestrial 
seismology \cite{nakata2019}. 
For instance, in helioseismology, the observed solar oscillations are 
stochastically driven by turbulent convective motions 
\cite{Dalsgaard2014,Gizon2005}, and the resulting signals are used to probe 
the solar interior \cite{Gizon2005,Pham2020Siam,Pham2025axisymreport,muller2024}.
On Earth, ambient noise can result from microseismic activity 
(and other natural or anthropogenic sources) \cite{nakata2019}. 
Ambient signals provide vast amounts of data because they only require 
sensors to monitor background vibrations. Consequently, passive imaging 
has seen a growing number of applications, including regional-scale 
imaging \cite{Shapiro2004,Shapiro2005}, as well as monitoring of glaciers 
\cite{Walter2020}, groundwater systems \cite{Voisin2017}, volcanoes 
\cite{Takano2020}, and applications in medical imaging \cite{douglas2017}.
In our work, we consider the ambient signals as superpositions of 
wavefields generated by stochastic sources.

In contrast to active-source imaging, which relies on 
dedicated excitation devices 
to probe the medium, passive imaging requires only sensors to record 
background vibrations. In addition, controlled-source acquisitions, while a 
standard practice in Earth exploration (e.g., \cite{Virieux2009}), are not 
feasible in asteroseismology. Nevertheless, the stochastic nature of 
ambient signals introduces additional complexity in data interpretation, 
and a central challenge is to relate these observations to deterministic 
objects.
To this end, the cross-correlation of ambient signals has become a 
fundamental tool. Early work by Aki \cite{aki1957} explored the use of 
autocorrelation techniques in seismology to investigate Earth's structure. 
In helioseismology, Duvall et al. \cite{duvall1993} demonstrated that the 
cross-correlation of random wavefields can be used to extract travel-time 
information. Weaver and Lobkis \cite{weaver2001} further established the 
connection between correlations of diffuse fields and the Green's function 
of the medium. 
In particular, cross-correlations exhibit a self-averaging property: under 
appropriate assumptions, long-time averages converge toward their expected 
value, enabling the extraction of deterministic information from stochastic 
recordings \cite{garnier2016}. 

This connection between the expected value of the cross-correlation of stochastic signals and the deterministic Green's function constitutes a cornerstone for interpreting ambient signals \cite{garnier2016,nakata2019}, in particular by 
enabling the extraction of the travel times.
% cf. applications in Earth imaging \cite{Shapiro2004,Shapiro2005,Gouedard2008,fichtner2016,fichtner2017} 
% and in helioseismology \cite{duvall1993,Gizon2005,muller2024}. 
For instance, in \cite{campillo2003}, empirical Green's functions 
are estimated from seismic-noise cross-correlations to retrieve 
wave travel times between seismic stations and subsequently perform
seismic tomography. 
This approach relies on the assumption that stochastic sources are 
uniformly distributed, which may not always hold \cite{curtis2008}. 
The extraction of travel times is also performed 
in \cite{garnier2009} to image 
reflectors in random media. Moreover, asymmetries in cross-correlations 
have been exploited to image the spatial distribution of ambient 
seismic sources in \cite{fichtner2016}.

In our work, we develop a quantitative reconstruction 
methodology for passive imaging based on the expected 
value of the cross-correlation, using the full wavefields 
without reducing to travel times. 
We consider the time-harmonic acoustic wave equation and 
model ambient noise as the superposition of waves generated 
by randomly distributed sources; that is, the randomness is 
introduced only through the source terms. 
We further assume that the random sources are spatially 
uncorrelated, \cite{garnier2016,nakata2019}.
Under this assumption, the expected value of the cross-correlation 
can be expressed as the solution of a deterministic wave equation 
whose right-hand side depends on the Green's function and the 
source covariance, e.g., \cite{fichtner2017}. 
We further ignore the `convenient source' assumption that is
used to further simplify the cross-correlation average so 
that it is proportional to the imaginary part of the Green's 
function (as used in, e.g., \cite{garnier2016,Pham2020Siam}).
This assumption is shown to be inappropriate in applications, cf.~\cite{curtis2008,fichtner2014}.
An approach close to ours is considered in \cite{fichtner2025}, 
which, however, relies on approximation of the Green's function, 
reduced to a single arrival with a given amplitude and travel time, 
rather than the full solution of
the wave equation. 
Cross-correlations are also employed in \cite{muller2024} but in the context of iterative holography which is a migration-type imaging approach based on the 
backpropagation of the data, and that does not necessarily update the Green's function.
Furthermore, \cite{muller2024} either uses closed-form Green's functions 
(with homogeneous backgrounds) or exploits spherical-harmonic decompositions 
for radially symmetric backgrounds to reduce the computational cost.

In our work, we instead compute the Green's function through a general numerical
discretization that applies to arbitrarily heterogeneous media, without relying on simplifying assumptions. We then consider the
resulting nonlinear inverse problem. We further extend this framework by considering the
first-order wave equations, which enables the computation of cross-correlations between
distinct fields, such as pressure and velocity components.

The quantitative inverse problem is formulated as the 
reconstruction of the model parameters given the expected values 
of the cross-correlations. The unknown parameters consist 
of the medium properties, namely the density and the wave 
speed for acoustic propagation, 
and the covariance of the stochastic sources. 
For the reconstruction, we adopt a minimization-based approach 
similar to full waveform inversion (FWI) \cite{Virieux2009,tarantola1984}, 
but using cross-correlations instead of wavefields. 
Starting from an initial set of parameters, a misfit function 
measuring the discrepancy between simulated and observed data 
is minimized iteratively. The procedure relies on a Newton-type 
method, in which the derivatives of the misfit functional are 
computed via the adjoint-state method 
\cite{chavent1974,plessix2006,Chavent2010,kern2016,Faucher2020}, 
which must now be adapted to the cross-correlation data.
In particular, there are now two adjoint states to be resolved, 
contrary to only one in FWI.

The numerical discretization for the forward and inverse problems is carried out 
using the Hybridizable Discontinuous Galerkin (HDG) method 
\cite{Cockburn2009,cockburn2010hybridizable,cockburn2016static,cockburn2023hybridizable}, 
following the implementation of \cite{Faucher2020} in the open-source software 
\texttt{Hawen}~\cite{Hawen2021}. The HDG framework is particularly well-suited for
first-order problems, as it provides direct access to both scalar and vector field 
quantities without enlarging the global linear system, and naturally supports 
$hp$-adaptivity and parallelisation. The adjoint-state formulations are derived 
consistently with both the HDG discretization and the structure of the cross-correlation 
problem. Linear systems are solved using the direct solver MUMPS \cite{Amestoy2001}, 
which efficiently handles multiple right-hand sides once the factorization is performed \cite{amlm:18a}. This is particularly convenient for the computations of the expected value of cross-correlations that involve the same operator with various right-hand sides.

In this paper, we first introduce the forward modeling problem in \cref{section:equations}, where we derive the expected value of the cross-correlation and show how it can be computed from the Green's functions.
In \cref{section:hdg-discretization}, we present the HDG discretization of the wave equation used to compute the expected cross-correlation numerically. 
% By working with the first-order system, we further detail the steps to compute the correlations between different fields.
The nonlinear inverse problem is formulated in \cref{section:inversion}, where we provide the framework and derive the adjoint-state method to compute the gradient of the misfit function adapted to the cross-correlation problem, while accounting for the HDG discretization.
Finally, in \cref{section:numerical}, we present numerical experiments that illustrate the method's feasibility and its performance compared to the more traditional case of controlled-source acquisition. 
 
% --------------------------------------------------------------------------
\section{Time-harmonic cross-correlation of acoustic waves}
\label{section:equations}
% --------------------------------------------------------------------------

In this section, we express the expected value of the cross-correlation 
between fields for the time-harmonic acoustic wave equation. 
Assuming the sources are uncorrelated in space, the expected value of 
the cross-correlation is computed using two consecutive resolutions of the 
wave problem, first with a Dirac source to obtain the Green's tensor, 
and then with a dense right-hand side. It gives access to the expected value 
of the cross-correlation between the position of the Dirac source and any 
other positions in the domain.

% ---------------------------------------------
\subsection{First-order wave equations}
% ----------------------------------------------

We consider the propagation of acoustic waves in a bounded 
domain in dimension $\dim$ (typically two or three), 
$\Omega \subset \mathbb{R}^\dim$.
The time-harmonic waves are given by the 
scalar pressure $p \in \mathbb{C}(\Omega)$ and the 
vectorial particle velocity $\bv\in \mathbb{C}(\Omega)^\dim$
that are solutions to (e.g., \cite{ColtonKress,faucher2017,Faucher2020}),
\begin{subequations} \label{eq:direct:wave_equation}
    \begin{empheq}[left=\empheqlbrace]{align}    
    &  \dfrac{\ii \omega}{\kappa(\bx,\omega)}\, p(\bx) - \nabla \cdot \bv(\bx) = f(\bx),\\
    &-\ii \omega \rho(\bx)\, \bv(\bx) + \nabla p(\bx) = \mathbf{0}.
    \label{eq:direct:wave_equation_b}
    \end{empheq}
\end{subequations}
Here, $\omega \in \mathbb{R}^+$ denotes the angular frequency, and
the scalar source term is denoted $f$.
The medium is characterized by the bulk modulus $\kappa$, 
and the density $\rho$. We allow for complex-valued, 
frequency-dependent $\kappa$ to account for attenuation, and 
further refer to \cite{Faucher2023viscoacoustic} for more 
details on viscoacoustic models.
As an alternative model parameter, the wave speed 
$c(\bx,\omega)$ is defined such that
\begin{equation} \label{eq:wavespeed}
   c(\bx,\omega) \,=\, \sqrt{\dfrac{\kappa(\bx,\omega)}{\rho(\bx)}} \,,
\end{equation}
where $\sqrt{\cdot}$ selects the complex upper half-plane, 
e.g., \cite{Elbau2017,Faucher2023viscoacoustic}.

\newcommand{\bcdir}{\Gamma_{\mathrm{d}}}
\newcommand{\bcabc}{\Gamma_{\mathrm{abc}}}
We assume that the boundary $\Gamma$ of the domain 
$\Omega$ is decomposed in two disjoint parts: 
$\bcdir$ and $\bcabc$. 
On $\bcdir$, we impose homogeneous Dirichlet condition 
in pressure (i.e., free-surface condition) and 
first-order absorbing condition (\cite{engquist1977absorbing})
is imposed on $\bcabc$: 
\begin{subequations}\begin{align}
&
\text{on $\bcdir$, free-surface condition: } \hspace*{5em}
p \,=\, 0 \,; \\
& \text{on $\bcabc$, absorbing condition: } \quad
-\ii\,\omega\,p / c \,+\, \, \partial_\bn\, p \,=\,0 \,\, \Rightarrow \,\,
c\,\rho \,\, \bv \cdot \bn \, - \, p \,\, = 0 \,.
\end{align} \end{subequations}
Here, $\bn$ denotes the normal direction and $\partial_\bn$ the normal derivative. This configuration is in particular 
representative of Earth media, \cite{faucher2017}.

We introduce the notation $\mathcal{L}$ 
such that \cref{eq:direct:wave_equation}
is written as,
\begin{equation}\label{eq:direct:wave_equation_compact}
    \mathcal{L}\,\mathcal{U}   \,=\, \mathcal{F}, 
    \qquad \text{with} \quad
    \mathcal{U} = 
    \begin{pmatrix}[1.25]
        p \\ \bv
    \end{pmatrix} \,,
    \quad
    \mathcal{F} =  
    \begin{pmatrix}[1.25]
    f \\ \mathbf{0}
    \end{pmatrix}
    \quad \text{and}\quad
% \end{equation}
% with the block operator
% \begin{equation}
    \mathcal{L} = 
    \begin{pmatrix}[1.25]
        \dfrac{\ii \omega}{\kappa} & 
        - \nabla \cdot \\[6pt]
       \nabla & -\ii \omega \rho
    \end{pmatrix}.
\end{equation}
Furthermore, we use the notation `$\comp$' 
to denote the set of components:
\begin{equation}
    \comp \,:=\, \{p, v_x, v_y, v_z \} \,\,\text{ in 3D,} \qquad
    \comp \,:=\, \{p, v_x, v_y\} \,\,\text{ in 2D.}
\end{equation}
Its cardinality is denoted by $\ncomp := \dim + 1$.

% ---------------------------------------------------------
\begin{definition}[Green's kernel] 
\label{definition:green_pressure}
The Green's kernel $\mathsf{G}$ is 
defined as the fundamental solution 
to
\begin{equation}
    \mathcal{L} \mathsf{G}( \bx, \by) = \begin{pmatrix} \delta_\by(\bx) \\ \mathbf{0}
    \end{pmatrix}, \qquad \mathsf{G} = \left ( \mathcal{G}^{[i]} \right )_{i \in \comp} \,,
\end{equation}
with the same boundary conditions as \cref{eq:direct:wave_equation} and 
where $\delta_\by$ is the Dirac distribution at position $\by$.
\end{definition}
Note that from \cref{eq:direct:wave_equation_b}, we have that,
\begin{equation}
\mathcal{G}^{[v_i]}(\bx,\by)
=
\frac{1}{\ii \omega \rho(\bx)}\,\partial_{i}\mathcal{G}^{[p]}(\bx,\by),
\qquad
\end{equation}
where $\partial_i$ is the partial derivative with 
respect to the $i^\mathrm{th}$ coordinate.

\begin{proposition}[Reciprocity]\label{prop:reciprocity}
  For the acoustic wave system \cref{eq:direct:wave_equation}, 
  the component of the Green's kernel associated with the 
  pressure field, 
  $\mathcal{G}^{[p]}$, is symmetric with respect to 
  $(\bx,\by)$ such that,
  \begin{equation} \label{eq:green-reciprocity}
  \forall\, (\bx, \by) \in \Omega^2,\qquad 
  \mathcal{G}^{[p]}(\bx, \by) = \mathcal{G}^{[p]}(\by,\bx) \,.
  \end{equation}
\end{proposition}

\begin{proposition}[Representation] \label{prop:representation}
The solution $(p, \bv)$ of \cref{eq:direct:wave_equation} for 
right-hand side $f$ can be represented with the Green's kernel
such that,
\begin{equation}
    p(\bx) = \int_\Omega \mathcal{G}^{[p]}(\bx, \by) f (\by) \rmd \by = \int_\Omega \mathcal{G}^{[p]}(\by, \bx) f (\by) \rmd \by \,,
\end{equation}
\begin{equation}
    v_i(\bx) = \dfrac{1}{\ii \omega \rho(\bx)} \partial_{i} p(\bx) =\dfrac{1}{\ii \omega \rho(\bx)} \int_\Omega \partial_{i} \mathcal{G}^{[p]}(\bx, \by) f(\by) \rmd \by = \int_{\Omega} \mathcal{G}^{[v_i]} (\bx, \by) f(\by) \rmd \by.
\end{equation}
In the following, the representation for 
the vector solution $\mathcal{U}=(p \,\,\, \bv)^t$ is written as, 
\begin{equation} \label{eq:Green_rep}
    \mathcal{U} = \int_\Omega \mathsf{G}(\bx, \by) f(\by) \rmd \by.
\end{equation}
\end{proposition}

% --------------------------------------------------
\subsection{Cross-correlation of ambient signals}
% --------------------------------------------------

% =====================================================
\subsubsection{Source covariance}
% =====================================================

We consider that the ambient signals corresponds 
to a superposition of waves excited from 
stochastic sources. Keeping the notation $f$ for 
the random variable, the source is described by its 
covariance such that,

\begin{equation}
\forall\, \psi,\varphi \in L^2(\Omega),
\qquad
\langle \Cov[f]\psi,\varphi \rangle
=
\mathbb{E}\!\left[
\langle f,\psi\rangle_{L^2(\Omega)}\,
\overline{\langle f,\varphi\rangle_{L^2(\Omega)}}
\right].
\end{equation}
The following assumption on the source covariance is a key-ingredient 
for further derivations; It is employed, for instance in~\cite{garnier2016},
in applications of Earth seismology \cite{fichtner2016} and 
helioseismology \cite{Pham2020Siam}. 
\begin{assumption} \label{src_assumption}
We assume that the source $f$ is spatially uncorrelated, such that the covariance operator $\Cov[f]$ is the multiplication operator by the function $\mathcal{S}\in L^\infty(\Omega)$.
\end{assumption}

In the present work, $\mathcal{S}$ is referred to as the \emph{source covariance}.
Since we consider a single scalar source field $f$, $\mathcal{S}$ is real-valued 
and non-negative, and coincides with the local variance density of $f$. In the seismological literature, it is also commonly referred to as the 
\emph{source power-spectral density} (see, e.g., \cite{fichtner2016}).

The present framework naturally extends to vector-valued source fields appearing 
on the right-hand side of \cref{eq:direct:wave_equation_b}. In that case, the 
full source covariance operator is matrix-valued: its diagonal entries 
$\mathcal{S}_{ii}$ correspond to the variance densities of the individual source 
components, while the off-diagonal entries $\mathcal{S}_{ij}$, $i \neq j$, 
describe the cross-covariances between distinct source components and may, in 
general, be complex-valued. In this setting, $\mathcal{S}$ coincides with the 
pressure-pressure entry $\mathcal{S}_{pp}$ of the full source covariance operator. 
Throughout this work, we restrict ourselves to the scalar case, and $\mathcal{S} 
= \mathcal{S}_{pp}$ always denotes this single variance density.

Note that under \cref{src_assumption}, since $\Omega$ is bounded,
$\mathcal{S}\in L^\infty(\Omega)$ also implies $\mathcal{S}\in L^2(\Omega)$.
The covariance operator admits the distributional kernel representation
\[
(\mathbf{x}_1,\mathbf{x}_2)\mapsto
\mathbb{E}[f(\mathbf{x}_1)\overline{f(\mathbf{x}_2)}]
=
\mathcal{S}(\mathbf{x}_1)\delta(\mathbf{x}_1-\mathbf{x}_2).
\]
Indeed, for all $\psi,\varphi\in L^2(\Omega)$,
\begin{equation}\label{eq:src_assumption}
\int_{\Omega}\int_{\Omega}
\mathbb{E}[f(\mathbf{y}_1)\overline{f(\mathbf{y}_2)}]
\psi(\mathbf{y}_1)\overline{\varphi(\mathbf{y}_2)}
\,d\mathbf{y}_1d\mathbf{y}_2
=
\int_{\Omega}
\mathcal{S}(\mathbf{y})
\psi(\mathbf{y})\overline{\varphi(\mathbf{y})}
\,d\mathbf{y}.
\end{equation}

% =====================================================
\subsubsection{Expected value of the cross-correlation}
% =====================================================

The expected value of the cross-correlation between 
two time-harmonic signals $u_1$ and $u_2$ measured 
respectively in $\bx_1$ and $\bx_2$ at frequency 
$\omega$ is defined as 
$\mathcal{C}_{u_1 u_2}(\bx_1, \bx_2, \omega)$ such that,
\begin{equation}
    \mathcal{C}^{[u_1 u_2]}(\bx_1, \bx_2, \omega)
    = \mathbb{E}\!\left[\,u_1(\bx_1, \omega)\,\overline{u_2(\bx_2, \omega)}\,\right].
\end{equation}
Consider that the signals are solutions to the wave equations \cref{eq:direct:wave_equation_compact}, 
we define the matrix $\mathsf{C}$ of the expected value of the cross-correlation such that,
omitting the frequency dependency for clarity:
\begin{equation} \label{eq:CC_def}
    \mathsf{C}(\mathbf{x}_1, \mathbf{x}_2)
    =
    \left( \mathcal{C}^{[ij]}(\bx_1, \bx_2) \right)_{(i,j) \in \comp^2} 
    \;=\;
    \mathbb{E}\!\left[\,\mathcal{U} (\mathbf{x}_1)\,
    \otimes \overline{{\mathcal{U}(\bx_2)}}
    \right].
\end{equation}

% This quantity represents the expected cross-correlation between the $j$-th component of the wavefield at $\mathbf{x}_1$ and the measured pressure field at $\mathbf{x}_2$.
Using \cref{prop:representation}, we have that:
\begin{equation}
    \mathsf{C}(\mathbf{x}_1, \mathbf{x}_2)
    = \int_{\by_1} \int_{\by_2} \mathbb{E}\!\left[\, f(\by_1)
    \,\, \overline{f(\by_2)} \right ] \mathsf{G}(\bx_1, \by_1) \otimes
     \overline{ \mathsf{G}(\bx_2, \by_2) }\rmd {\by_1} \rmd{\by_2},
\end{equation}
and with \cref{eq:src_assumption}, this gives,

\begin{equation}
    \mathsf{C}(\mathbf{x}_1, \mathbf{x}_2)
    = \int_\by \mathcal{S}(\by)\mathsf{G}(\bx_1, \by)
   \otimes \mathsf{G}(\bx_2, \by) \rmd \by \,.
\end{equation}

To further simplify this quantity, one can use 
the reciprocity property of the Green's function
\cref{prop:reciprocity}. 
Assuming one of the correlated field is the pressure 
field, we denote by $\mathsf{C}^{[p]}$ the resulting vector quantity, that corresponds to the first 
column of $\mathsf{C}$:
\begin{equation}
    \mathsf{C}^{[p]}(\mathbf{x}_1, \mathbf{x}_2)
    = \int_\by \mathsf{G}(\bx_1, \by)
    \mathcal{S}(\by) \overline{\mathcal{G}^{[p]}( \bx_2, \by)} \rmd \by \,.
\end{equation}
Using the reciprocity from \cref{prop:reciprocity}, we get,
\begin{equation}
    \mathsf{C}^{[p]}(\mathbf{x}_1, \mathbf{x}_2)
    = \int_\by \mathsf{G}(\bx_1, \by)
    \mathcal{S}(\by) \overline{\mathcal{G}^{[p]}( \by, \bx_2)} \rmd \by\,.
\end{equation}
that can be written as
\begin{equation}\label{eq:CC_main}
    \mathsf{C}^{[p]}(\mathbf{x}_1, \mathbf{x}_2) = \int_\Omega \mathsf{G}(\bx_1, \by) \mathcal{M}(\by,\bx_2) \rmd \by \,, \hspace*{1em}
    \text{with } \quad
    \mathcal{M}(\by,\bx_2) \,=\, \mathcal{S}(\by) \, 
    \overline{\mathcal{G}^{[p]}(\by,\bx_2)}.
\end{equation}
We recognize an expression similar to the 
representation formula \cref{eq:Green_rep}.
Therefore,  
$\mathsf{C}(\bx_1, \bx_2)$ can be computed 
as the solution to the forward problem \cref{eq:direct:wave_equation} 
with source term $\mathcal{M}$.
This prevents us from having to formally compute the integration, however, it requires the reciprocity of the forward operator
Green's function \cref{eq:green-reciprocity} to compute $\mathcal{G}^{[p]}(\bx_2, \by)$ only once, with the Dirac source in $\bx_2$.

\begin{remark}[Convenient source assumption]

Additional assumptions on the nature of the source covariance can be incorporated to further simplify the expression. Within a convenient source assumption, the expected value 
of the cross-correlation is proportional to the imaginary 
part of the Green's function, e.g., 
\cite{garnier2016,weaver2001,Pham2020Siam},
\begin{equation}
%%  \forall (\bx_1, \bx_2) \in \Omega, \;
    \mathcal{C}^{[p]}(\bx_1 , \bx_2) \propto \Im  \left (\mathcal{G}^{[p]}( \bx_1 , \bx_2) \right )\,,
    \qquad \text{convenient source assumption}.
\end{equation}
Such a result can be obtained by relating the source cross-covariance to the imaginary part of the model parameters, 
i.e., using the attenuation, see, e.g., \cite[Appendix A]{Pham2020Siam}. 
This assumption simplifies the computations but is not realistic in several applications, e.g., \cite{curtis2008,tsai2009,fichtner2014}. It is not considered in our work.
\end{remark}

% -------------------------------------------------------------
\subsection{Computation of the cross-correlation} \label{section:twosteps}
% -------------------------------------------------------------

For the positions of the signals used in the 
cross-correlations $\mathsf{C}(\bx_1,\bx_2)$, 
we introduce $\rcvs$ and $\srcs$, which correspond to the set of 
positions for $\bx_1$ and $\bx_2$, respectively. 
Note that these sets are the same in practice
(as it only correspond to the positions of the receiver devices)
but they are considered distinct here for numerical flexibility. 
The positions in $\srcs$ are further referred to 
as \emph{virtual sources} as the cardinality of 
$\srcs$ defines the number of computation of Green's function,
so we typically consider $\srcs \subset \rcvs$. 
On the other hand, the cardinality of $\rcvs$ only affects 
the number of evaluations of the computed fields.
\cref{algorithm:forward-cc} provides an outline of the 
numerical steps for the computation of the cross-correlation
average.

\begin{algorithm}[ht!]
\caption{The computation of the expected value of the cross-correlation
         assuming that the sources are uncorrelated in space is performed
         by solving two wave problems consecutively.
         }
\label{algorithm:forward-cc}

\begin{algorithmic}
% ko \renewcommand{\arraystretch}{1.5}

\STATE \noindent \makebox[4.0em][l]{\textbf{Inputs:}} 
\hspace{0.50em} Domain $\Omega$ with 
             model parameters $c$, 
             $\rho$ and source covariance 
              $\mathcal{S}$; 
             frequency $\omega$; 
             set of positions $\srcs\subset\Omega$ and $\rcvs\subset\Omega$;  
\hspace{1em} 
\vspace*{0.50em}

\FOR{$\bx_2 \in \srcs$}
\vspace*{0.50em}

 % \STATE $\sharp$ \textit{We solve the wave problem with four right-hand sides to obtain the $4\times4$ solution  $\mathcal{U}_G$}
  \STATE \makebox[26em][l]{$\bullet\,\,$
         Solve the wave problem \cref{eq:direct:wave_equation} 
         with a Dirac source in $\bx_2$:}
         $\mathcal{L} \, \begin{pmatrix}
         u_p \\ \mathbf{u_v}             
         \end{pmatrix} \,=\, \begin{pmatrix}
             \delta_{\bx_2} \\ \mathbf{0}
         \end{pmatrix}.
         $
\vspace*{0.40em}

 % \STATE $\sharp$ \textit{We create the new right-hand side $\widehat{\mathcal{F}}$ of size $4\times4$}
  \STATE \makebox[26em][l]{$\bullet\,\,$
         Compute the RHS ${\mathcal{M}}$
         from \cref{eq:CC_main}:}
         ${\mathcal{M}} \,=\, 
           \mathcal{S}\, 
           \overline{u_p}$.
  \vspace*{0.40em}

  \STATE \makebox[26em][l]{$\bullet\,\,$
         Solve the wave problem \cref{eq:direct:wave_equation} 
         with the source ${\mathsf{M}}$:}
         $\mathcal{L} \, \mathsf{U}_C \,=\, {\mathcal{M}}$. 
  \vspace*{0.40em}
  
  \STATE $\bullet\,\,$ $\mathcal{C}(\bx_1,\bx_2)$ is given by $\mathsf{U}_C(\bx_1)$, for all $\bx_1\in \rcvs$.
  \vspace*{0.40em}
  
\ENDFOR
\end{algorithmic}
\end{algorithm}

\begin{remark}\label{remark:sources}

In \cref{algorithm:forward-cc}, a loop over each 
position $\bx_2$ is introduced. For each of them, 
the wave problem must be solved twice (once with a Dirac 
source and once with the source term $\mathcal{M}$). 
However, in practice, we use the direct solver MUMPS to 
solve the linear system, and we rely on its multiple 
right-hand sides (RHS) feature, \cite{Amestoy2001,amlm:18a}, 
such that, once the matrix factorization is performed, all RHS 
can be solved at a minimal cost. 
Therefore, instead of a loop over the positions $\bx_2$, RHS 
of dimension $\mathrm{Card}(\srcs)$ are assembled, and the system 
is solved in a single resolution step. 

\end{remark}
 
% --------------------------------------------------------------------------
\section{Hybridizable Discontinuous Galerkin discretization}
\label{section:hdg-discretization}
% ----------------------------------------------------------------------

The wave equation~\cref{eq:direct:wave_equation} is discretized using the Hybridizable Discontinuous Galerkin (HDG) method. This approach enables the solution of the first-order system, thereby providing access to both the pressure $p$ and the velocity $\bv$ without any post-processing, while leading to a global linear system involving a single scalar unknown \cite{Faucher2020}. We follow the derivation of the HDG discretization of the acoustic equation~\cref{eq:direct:wave_equation} presented in \cite{Faucher2020}, and we outline below the modifications required for the computation of the covariance.

\subsection{Notations}
\paragraph{Domain discretization}
The domain $\Omega$ is discretized using a non-overlapping partition and we denote with $\mathcal{T}_h$ the mesh composed of $N$ elements $K_e$ such that,
\begin{equation}
    \mathcal{T}_h = \bigcup_{e=1}^N K_e.
\end{equation}

The set of $N_\Sigma$ mesh faces is denoted by $\Sigma$. It is composed of $N_{\Sigma^I}$ interior faces $\Sigma^I$, that are shared by two elements, and $N_{\Sigma^E}$ exterior faces $\Sigma^E$, that lie on the boundary of the computational domain. The set of exterior faces is further decomposed into two disjoint subsets: the subset $\Sigma^E_{\mathrm{ABC}}$, on which absorbing boundary conditions are imposed, and the subset $\Sigma^E_{\mathrm{D}}$, where Dirichlet boundary conditions are prescribed.
In the numerical implementation, simplex elements are employed, namely triangle elements in two  dimensions and tetrahedra in three dimensions.

\paragraph{Function spaces}
To approximate the solution of \cref{eq:direct:wave_equation}, we consider 
piecewise-polynomial functions. Let $\mathbb{P}_\mathfrak{p}(K)$ denote the 
space of polynomials of degree lower than or equal to $\mathfrak{p}$ defined 
on the element $K$. We define, in dimension $\dim$,
\begin{subequations}
\begin{align}
  W_h &= \left\{ w_h \in L^2(\Omega)\; :\;
       w_h|_{K_e} \in \mathbb{P}_{\mathfrak{p}_e}(K_e),\; 
       \quad \forall K_e \in \mathcal{T}_h \right\}, \\
  \mathbf{W}_h &= \left\{ \mathbf{w}_h \in (L^2(\Omega))^\dim\; :\;
        \mathbf{w}_h|_{K_e} \in \big(\mathbb{P}_{\mathfrak{p}_e}(K_e)\big)^\dim,\; \quad
       \forall K_e \in \mathcal{T}_h \right\}, \\
    U_h &= \left \{ u_h \in L^2(\Sigma) \; : \; u_h|_{\mathfrak{f}} \in 
    \mathbb{P}_{\mathfrak{q}_k}(\mathfrak{f}), \quad \forall \mathfrak{f} \in \Sigma \right \}.
\end{align}
\end{subequations}
Each of these spaces is equipped with Lagrange basis functions, denoted by 
$\{\varphi_{e,j}\}_{j \in \{ 1,\ldots, n_{\mathrm{dof}}^e \}}$ 
on each element $K_e \in \mathcal{T}_h$, and 
$\{\xi_{\mathfrak{f},j}\}_{j \in \{ 1,\dots,  n_{\mathrm{dof}}^\mathfrak{f} \}}$ 
on each face $\mathfrak{f} \in \Sigma$.
\paragraph{Jump operator} 
On an interface $\mathfrak{f}$ shared by two elements $K^{+}$ and $K^{-}$, 
the jump of a vector $\bv$ is denoted by double brackets, 
$\llbracket \bv \rrbracket$, such that,
\begin{equation} \label{equation:jump}
\llbracket \bv \cdot \boldsymbol{n} \rrbracket_{\mathfrak{f}}
:= {\mathbf{v}^{+}} \cdot \boldsymbol{n}^{+}
 + {\mathbf{v}^{-}} \cdot \boldsymbol{n}^{-}
 = {\mathbf{v}}^{+} \cdot \boldsymbol{n}^{+}
 - {\mathbf{v}^{-}} \cdot \boldsymbol{n}^{+},
\end{equation}
where $\bv^{\pm}$ denotes $\bv$ taken on the cell $K^{\pm}$, 
and $\boldsymbol{n}^{\pm}$ is the outward-pointing normal 
vector to $\mathfrak{f}$ from $K^{\pm}$.

% -------------------------------------------------------------
\subsection{Hybridizable Discontinuous Galerkin discretization}
% --------------------------------------------------------------

\paragraph{Strong form}
For the system of equations \cref{eq:direct:wave_equation}, the HDG formulation writes as: 
find $(p_h,\mathbf{w}_h,\lambda_h)
\in W_h \times \mathbf{W}_h \times U_h $
such that on each cell $K_e \in \mathcal{T}_h$,
\begin{subequations} 
\begin{empheq}[left=\empheqlbrace]{align}
\dfrac{\ii \omega }{\kappa} p_h  - \nabla \cdot \bv_h &= f, 
& \text{on } K_e,  \\[4pt]
-\ii \omega \rho \bv_h + \nabla p_h  &= 0, &  \text{on } K_e,  \\[4pt]
p_h &= \lambda_h, \label{eq:HDG_strong_face}
& \text{on } \partial K_e, 
\end{empheq}
\end{subequations}

\begin{align}
\text{On each interior face } \mathfrak{f} \in \Sigma^i:&  \qquad
\llbracket \widehat{\bv_h} \cdot \mathbf{n}_{ \mathfrak{f}} \rrbracket = 0, \label{eq:HDG_strong_cont}
\\[6pt]
\text{On each boundary face }  \mathfrak{f} \in \Sigma^e: & \qquad
\begin{cases}
    -\frac{1}{c\rho}\, \lambda_h + \widehat{\bv_h} \cdot \mathbf{n} = 0, \quad &\text{if } \mathfrak{f} \subset \Gamma_\mathrm{abc} \\
    \lambda_h = 0, & \text{if } \mathfrak{f} \subset \Gamma_{\mathrm{d}}   \label{eq:HDG_strong_bc}
\end{cases}
\end{align}
where $\widehat{\bv_h}$ stands for the numerical trace of $\bv_h$.
In the HDG formulation, it is expressed in terms of the trace unknown $\lambda_h$ and 
a stabilization parameter $\tau$, e.g., \cite{Cockburn2009, pham2024, Faucher2020}:

\begin{equation}
    \forall \mathfrak{f} \in \Sigma, \qquad \widehat{\bv_h} \cdot \bn 
    = \bv_h \cdot \bn \, - \, \tau ( p_h - \lambda_h). \label{eq:HDG_stab}
\end{equation}
In the numerical experiments, the stabilization parameter is taken to be $\tau = 1/\rho_0$, where $\rho_0 = \rho(\mathbf{x}_{\mathfrak{f}})$ denotes the density evaluated on the face $\mathfrak{f}$, 
\cite{Faucher2020,nguyen2009implicit}.

\paragraph{Weak form}

To obtain a weak formulation, the strong form is integrated against test functions $ (\varphi_h, \boldsymbol{\psi}_h) \in W_h \times \mathbf{W}_h$
and then we integrate by parts. Replacing the expression of the numerical traces according to \cref{eq:HDG_stab}, we obtain, for each element $K_e \in \mathcal{T}_h$, \cite[Eq. (29)]{Faucher2020}:
\begin{subequations}
    \begin{empheq}[left=\empheqlbrace]{align}
    \ii \omega \int_{K_e} \dfrac{1}{\kappa} p_h \varphi_h +\int_{K_e} \bv_h \cdot \nabla \varphi_h  - \int_{\partial K_e } \big ( \bv_h \cdot \bn - \tau ( p_h - \lambda_h )  \big ) \varphi_h  &= \int_{K_e} f \varphi_h , \\
    - \ii \omega \int_{K_e} \rho \bv_h \cdot \boldsymbol{\psi}_h  - \int_{K_e} p_h \nabla \cdot \boldsymbol{\psi}_h + \int_{\partial K_e} \lambda_h \boldsymbol{\psi}_h \cdot \bn &=0.
    \end{empheq} \label{eq:HDG_weak_local1}
\end{subequations}

We perform a reverse integration by part on the 
first equation to obtain,
\begin{subequations}
    \begin{empheq}[left=\empheqlbrace]{align}
    \ii \omega \int_{K_e} \dfrac{1}{\kappa} p_h \varphi_h
      - \int_{K_e} (\nabla \cdot \bv_h)\, \varphi_h
      + \int_{\partial K_e} \tau \left( p_h - \lambda_h \right) \varphi_h
      &= \int_{K_e} f \varphi_h , \\
    - \ii \omega \int_{K_e} \rho \bv_h \cdot \boldsymbol{\psi}_h  - \int_{K_e} p_h \nabla \cdot \boldsymbol{\psi}_h + \int_{\partial K_e} \lambda_h \boldsymbol{\psi}_h \cdot \bn &=0.
    \end{empheq} \label{eq:HDG_weak_local}
\end{subequations} 
On the faces of the elements, 
\cref{eq:HDG_strong_face,eq:HDG_strong_bc,eq:HDG_strong_cont}, 
are integrated against test functions $\xi_h \in U_h$, 
and we obtain the following.
\begin{itemize}
    \item On interior faces, $\mathfrak{f} \in \Sigma^i$, 
          from \cref{eq:HDG_stab,eq:HDG_strong_cont} we have,
    \begin{equation}
        \int_{\mathfrak{f}} 
        \llbracket \bv_h \cdot \bn - \tau \left( p_h - \lambda_h \right) \rrbracket 
        \, \xi_h \, = 0,
        \qquad \forall \xi_h \in U_h. \label{eq:HDG_weak_int}
    \end{equation}
    
    \item On the boundary, we distinguish the faces according to the prescribed boundary conditions. 
    For absorbing boundary conditions, we use \cref{eq:HDG_stab} such that, 
    \begin{equation}
        \forall \mathfrak{f} \in \Sigma^e \cap \Gamma_{\mathrm{ABC}}, \quad
        \int_{\mathfrak{f}} 
        \left( 
            \bv_h \cdot \bn 
            - \tau \left( p_h - \lambda_h \right) 
            - \dfrac{1}{\rho c} \lambda_h 
        \right)
        \xi_h \, 
        = 0,
        \qquad \forall \xi_h \in U_h. \label{eq:HDG_weak_extABC}
    \end{equation}
    
    \item The Dirichlet conditions are weakly imposed:
    \begin{equation}
        \forall \mathfrak{f} \in \Sigma^e \cap \Gamma_{\mathrm{D}}, \quad
        \int_{\mathfrak{f}} 
        \lambda_h \, \xi_h \,
        = 0,
        \qquad \forall \xi_h \in U_h. \label{eq:HDG_weak_extDir}
    \end{equation}
\end{itemize}
The system composed of \cref{eq:HDG_weak_local,eq:HDG_weak_int,eq:HDG_weak_extABC,eq:HDG_weak_extDir} 
forms the weak formulation of the HDG discretized problem. 

% -----------------------------
\paragraph{Matrix representation}
% -----------------------------

The solutions are represented in a Lagrange polynomial basis with 
$n_{\mathrm{dof}}^e$ degrees of freedom on each element $K_e$, 
denoted by $\{\varphi_{e,k}\}_{k=1}^{n_{\mathrm{dof}}^e}$. 
On the element $K_e$, the coefficients of the unknowns 
$\bv_h|_{K_e}$ and $p_h|_{K_e}$ are denoted by $\mathrm{u}_{e,k}^{[\bullet]}$, 
with $\bullet \in \comp$ and $k = 1, \ldots, n_{\mathrm{dof}}^e$. We have,
\begin{equation} \label{eq:lagrange_rep}
    p_h|_{K_e} 
    = \sum_{k=1}^{n_{\mathrm{dof}}^e} 
    \mathrm{u}_{e,k}^{[p]} \, \varphi_{e,k}, \quad
    v_{x,h}|_{K_e} 
    = \sum_{k=1}^{n_{\mathrm{dof}}^e} 
    \mathrm{u}_{e,k}^{[v_x]} \, \varphi_{e,k}, \quad
    \text{and similarly for } v_{y,h}, v_{z,h}.
\end{equation}
The concatenation of the solution coefficients on element $K_e$ is denoted by
\begin{equation} \label{eq:hdg-unknown-representation}
    \boldsymbol{\mathrm{U}}_e 
    = \Big( \mathrm{U}_e^{[v_x]} \; \mathrm{U}_e^{[v_y]} \; \mathrm{U}_e^{[v_z]} \; \mathrm{U}_e^{[p]} \Big)^{\mathrm{t}}, 
    \qquad \text{with}
    \quad \mathrm{U}_e^{[\bullet]} = \Big ( \mathrm{u}_{e,k}^{[\bullet]} \Big)_{k = 1, \ldots, n_{\mathrm{dof}}^e},
    \quad \bullet \in \comp,
\end{equation}
where the notation 
$(u_k)_{k=1,\ldots,n}$ denotes a column vector of size $n$.
The HDG unknown $\lambda_h$ (which approximates the trace of the scalar pressure field) is defined on the faces of the element such that
\begin{equation}
    \lambda_h|_{\mathfrak{f}} = \sum_{k=1}^{n_{\mathrm{dof}}^\mathfrak{f}} 
    \mathrm{q}_{\mathfrak{f},k} \, \xi_{\mathfrak{f},k}.
\end{equation}
Its coefficients are gathered in a global vector $\Lambda$:
\begin{equation}
    \Lambda \,=\, 
    \left( \boldsymbol{\mathrm{q}}_{\mathfrak{f}} \right)_{\mathfrak{f}\in \Sigma}
    \, , \qquad \text{with} \quad
    \boldsymbol{\mathrm{q}}_{\mathfrak{f}} \,=\, 
    \left( \mathrm{q}_{\mathfrak{f},k} \right)_{k=1,\ldots,n_{\mathrm{dof}}^\mathfrak{f}} \,.
\end{equation}
We also introduce the restriction matrix $\mathcal{R}_e$, which 
extracts from $\Lambda$ the coefficients associated 
to the faces of element $K_e$:
\begin{equation}
    \mathcal{R}_e \, \Lambda \,=\, 
    \left( \boldsymbol{\mathrm{q}}_{\mathfrak{f}} \right)_{\mathfrak{f}\in \partial K_e} \,.
\end{equation}

Using the discretized representation of the unknowns, the equations \cref{eq:HDG_weak_local,eq:HDG_weak_int,eq:HDG_weak_extABC,eq:HDG_weak_extDir} 
are written in matrix form such that,
cf.~\cite[Eq.~(3.24)]{pham2024}, \cite[Eqs.~(35), (38)]{Faucher2020},
\begin{subequations}\label{eq:HDG_mat_full}
\begin{empheq}[left=\empheqlbrace]{align}
    \mathbb{A}_e \boldsymbol{\mathrm{U}}_e + \mathbb{D}_e \mathcal{R}_e \Lambda 
    &= \mathbb{F}_e, \hspace{2em} \forall K_e \in \mathcal{T}_h, \label{eq:HDG_mat_loc}\\
    \sum_{e=1}^{n_\text{cell}} 
    \mathcal{R}_e^\mathrm{t} 
    \left( 
        \mathbb{B}_e \brmU_e 
        + \mathbb{L}_e \mathcal{R}_e \Lambda 
    \right) &= 0, \label{eq:HDG_mat_glob}
\end{empheq}
\end{subequations}
where we refer to \cite[Section 3.4]{Faucher2020} for the 
explicit description of the local HDG matrices $\mathbb{A}_e$,
$\mathbb{D}_e$, $\mathbb{B}_e$, $\mathbb{L}_e$ that  
involve the integration of the basis functions on the element.
Here, the right-hand side $\mathbb{F}_e$ is given by,
\begin{equation}\label{eq:hdg-rhs}
    \mathbb{F}_e 
    = \begin{pmatrix}
        \mathrm{F}_e & \mathbf{0}
    \end{pmatrix}^{\mathrm{t}},
    \qquad
    \text{with }
    \mathrm{F}_e
    = \big( 
        \langle f , \varphi_{e,k} \rangle_{K_e} 
      \big)_{k = 1, \ldots n_{\mathrm{dof}}^e},
\end{equation}
where $\langle \cdot, \cdot \rangle_{K_e}$ denotes the $L^2$ inner product 
on $K_e$.

To solve \cref{eq:direct:wave_equation}, we first eliminate 
the local unknowns $\brmU_e$ from \cref{eq:HDG_mat_glob} using 
\cref{eq:HDG_mat_loc}, and solve the following linear system 
for the global unknown $\Lambda$:
\begin{equation}
\mathbb{K} \, \Lambda \, = \, \mathbb{H} \, \quad\text{ with }\quad
\mathbb{K} \,=\, \sum_{e=1}^{n_\text{cells}} 
    \mathcal{R}_e^\mathrm{t} 
    \left(
        \mathbb{L}_e 
        \,-\, 
        \mathbb{B}_e \mathbb{A}_e^{-1} \mathbb{D}_e 
    \right) \mathcal{R}_e   \,;\quad
\mathbb{H} \,=\, 
    \,-\, 
    \sum_{e=1}^{n_\text{cell}} 
    \mathcal{R}_e^\mathrm{t}
    \left(  
        \mathbb{B}_e \mathbb{A}_e^{-1}
            \mathbb{F}_e \right) \,. \label{eq:HDG_mat}
\end{equation}
Therefore, the HDG forward solver works in two steps: first
the global unknown $\Lambda$ is solved from the
linear system that only involves degrees of freedom on the 
faces of the element. Then, the volume solution $\brmU_e$ 
is recovered by solving local (on each element) systems 
with \cref{eq:HDG_mat_loc}. Note that this operation is 
embarrassingly parallel and involves the resolution of relatively
small, dense systems.

\subsection{Computation of the components of the cross-correlation} \label{section:HDG_CC}
% --------------------------------------------------------------------

The computation of the cross-correlation involves 
the consecutive resolution of two problems with
different RHS, as given in \cref{section:twosteps}. 
Here, we describe their construction and the problems 
solved for the HDG discretization.

\paragraph{Step 1: Computation of the Green's function}
The first problem employs Dirac sources to obtain 
the Green's function $\mathsf{G}$, see \cref{algorithm:forward-cc}.
Let us denote the  HDG discretized unknowns for a given
column $j \in \comp$ of the Green's tensor 
as ($\boldsymbol{\mathrm{G}}_e$, $\Lambda_G$). 
Similar to \cref{eq:hdg-unknown-representation},
$\boldsymbol{\mathrm{G}}_e$ is composed of four blocks:
\begin{equation}
\boldsymbol{\mathrm{G}}_e
=
\big(
    \mathrm{G}_e^{[i]}
\big)_{ i \in \comp} \,. \end{equation}
The couple $(\mathrm{G}_e, \Lambda_G)$ 
solves the discretized HDG system \cref{eq:HDG_mat_full}:
\begin{subequations} \label{eq:HDG_green}
\begin{empheq}[left=\empheqlbrace]{align}
 &\quad  \mathbb{A}_e {\boldsymbol{\mathrm{G}}}_e
    + \mathbb{D}_e \mathcal{R}_e \Lambda_G
    \,=\, \mathbb{P}_e,
    \qquad \forall K_e \in \mathcal{T}_h, 
    \label{eq:HDG_greenA} \\
 &\quad  \sum_{e=1}^{n_{\mathrm{cell}}} 
    \mathcal{R}_e^\mathrm{t} 
    \left(
        \mathbb{B}_e {\boldsymbol{\mathrm{G}}}_e 
        + \mathbb{L}_e \mathcal{R}_e \Lambda_G
    \right)
    \,=\, \mathbf{0}\, \,.
\end{empheq}
\end{subequations}
Each RHS $\mathbb{P}_e$ 
has four blocks, cf.~\cref{eq:hdg-rhs}, and is given by 
\begin{equation}\label{eq:hdg-rhs-step1}
\mathbb{P}_e \,=\,
\begin{pmatrix}
\mathrm{P}_e & \mathbf{0}
\end{pmatrix}^\mathrm{t}, 
\qquad \text{with} \quad
\mathrm{P}_e \,=\,
\begin{cases}
    \big( \varphi_{e,k}(\bx_2)    \big)_{k = 1, \ldots, \ndof_e}
    \quad \text{if } \bx_2 \in K_e, \\
0 \,\, \text{otherwise}.
\end{cases}
\end{equation}
From the coefficients
$\mathrm{G}_e^{[i]}$, the approximation of the Green's 
function on cell $K_e$ is denoted by $\widehat{\mathrm{G}}_e^{[i]}$ and is given by
\begin{equation} \label{eq:G_approx}
    G_h^{[i]}(\bx) \mid_{K_e} \,=\, \widehat{\mathrm{G}}_e^{[i ]}(\bx) 
    \,=\, 
    \sum_{k=1}^{n_{\mathrm{dof}}^e} \, 
    \mathrm{G}_{e,k}^{[i]} \, \varphi_k(\bx) \,.
\end{equation}
% ====================================
\paragraph{Step 2: Computation of the 
expected value of cross-correlations}
% ====================================

We denote by ($\boldsymbol{\mathrm{C}}_e$, $\Lambda_C$)
the HDG discretized unknowns corresponding the expected 
value of the cross-correlation.
They solve the HDG system, 
\begin{subequations} \label{eq:HDG_cc}
\begin{empheq}[left=\empheqlbrace]{align}
 &\quad  \mathbb{A}_e {\boldsymbol{\mathrm{C}}}_e
    + \mathbb{D}_e \mathcal{R}_e \Lambda_C
    \,=\, \mathbb{M}_e,
    \qquad \forall K_e \in \mathcal{T}_h, \label{eq:HDG_ccA}\\
 &\quad  \sum_{e=1}^{n_{\mathrm{cell}}} 
    \mathcal{R}_e^\mathrm{t} 
    \left(
        \mathbb{B}_e {\boldsymbol{\mathrm{C}}}_e 
        + \mathbb{L}_e \mathcal{R}_e \Lambda_C
    \right)
    \,=\, \mathbf{0}\, \,.
\end{empheq}
\end{subequations}
The RHS $\mathbb{M}_e = \begin{pmatrix}
    \mathrm{M}_e & \mathbf{0}
\end{pmatrix}^{\mathrm{t}}$
is given by, see \cref{eq:CC_main},
for $l = 1, \dots, \ndof_e$,
\begin{equation}\label{eq:hdg-rhs-step2}
\mathrm{M}_{e,l}
\,=\, \int_{K_e}
\mathcal{S} (\bx) \,
\overline{\widehat{\mathrm{G}}_e^{[p]}} \varphi_l(\bx)
\,=\, \int_{K_e}
 \mathcal{S}(\bx) \, \varphi_l(\bx)
 \sum_r \, \overline{\mathrm{G}_{e,r}^{[p]}}\,
 \, \varphi_r(\bx) \,,
\end{equation}
and $\mathrm{M}_e = (\mathrm{M}_{e,l})_{l=1}^{\ndof_e}$.

\cref{algorithm:forward-cc-hdg} outlines the computational procedure within the HDG discretization framework. In this algorithm, the operations involving loops over $\bx_2$ are divided into two separate stages rather than being executed consecutively in a single loop. This choice is motivated by the multiple RHS feature of the direct solver, see \cref{remark:sources}, which enables efficient computation. Specifically, all solutions $\brmG_e$ can be computed simultaneously by assembling the corresponding RHS vectors into a single system. A similar strategy is employed for $\brmC_e$, allowing all associated solutions to be obtained in a single solver call.

\begin{algorithm}[ht!]
\caption{
Computation of the expected value of the cross-correlation 
in the HDG discretization framework. The algorithm is structured 
with two consecutive parts with loops over $\srcs$ and $\ncomp$ to emphasize the use of the solver with multiple right-hand sides, 
such that once the global matrix $\mathbb{K}$ is factorized, only 
two solver calls are required.
}
\label{algorithm:forward-cc-hdg}

\begin{algorithmic}[ht!]

\STATE{\# \textit{Prepare the HDG matrices}}
\FOR{$e \in n_{\mathrm{cells}}$}
  \STATE Build HDG discretized matrices: $\mathbb{A}_e$, $\mathbb{B}_e$, $\mathbb{D}_e$, $\mathbb{L}_e$
\ENDFOR
\STATE Assemble the HDG global matrix $\mathbb{K}$ and perform its factorization.
\vspace*{0.50em}

\STATE{\# \textit{Part 1: solve for the Green's kernel}}
\FOR{$\bx_2 \in \srcs$}
  % -------------------------------------------
% \FOR{$j \in 1,\ldots \ncomp$}
  \vspace*{0.20em}

  \STATE \makebox[12.5em][l]{Build local RHS matrix:} $\mathbb{P}_e$ from \cref{eq:hdg-rhs-step1};
  \STATE \makebox[12.5em][l]{Assemble global RHS:} 
         $\mathbb{H}_G \,=\, \,-\, \sum_{e=1}^{n_\text{cell}}  
                                   \mathcal{R}_e^\mathrm{t} 
                                   \left(\mathbb{B}_e \mathbb{A}_e^{-1} \mathbb{P}_e 
                                   \right) \,$; 
  \STATE \makebox[12.5em][l]{Solve HDG global system} $\mathbb{K} \, \Lambda_G \,=\, \mathbb{H}_G$
  \STATE Retrieve the local solutions $\brmG_e$ from $\Lambda_G$ with \cref{eq:HDG_greenA}.
  \vspace*{0.20em}
  
% \ENDFOR
\ENDFOR
  \vspace*{0.50em}
  % -------------------------------------------

\STATE{\# \textit{Part 2: solve for the Cross-correlation}}
\FOR{$\bx_2 \in \srcs$}
  % -------------------------------------------
% \FOR{$j \in 1,\ldots \ncomp$}
  \vspace*{0.20em}

  \STATE \makebox[12em][l]{Build local RHS matrix:} $\mathrm{M}_e$ from \cref{eq:hdg-rhs-step2}
  \STATE \makebox[12em][l]{Assemble global RHS:} 
         $\mathbb{H}_C \,=\, \,-\, \sum_{e=1}^{n_\text{cell}}  
                                   \mathcal{R}_e^\mathrm{t} 
                                   \left(\mathbb{B}_e \mathbb{A}_e^{-1} \mathbb{M}_e 
                                   \right) \,$;
  \STATE \makebox[12em][l]{Solve HDG global system:} $\mathbb{K} \, \Lambda_C \,=\, \mathbb{H}_C$;
  \STATE Retrieve the local solutions $\brmC_e$ from $\Lambda_C$ with \cref{eq:HDG_ccA}
  \vspace*{0.20em}
% \ENDFOR
\ENDFOR
% \FOR{$\bx_1 \in \sum_{rcv}$}
% \STATE Compute and store the expected cross-correlation value as  \note{?}
% \ENDFOR
\end{algorithmic}
\end{algorithm}

% --------------------------------------------------------------------------
\section{Quantitative nonlinear inverse problem}
\label{section:inversion}

In this section, we describe the methodology for the nonlinear quantitative 
inversion. The model parameters characterising the wave problem, namely the 
wave speed $c$, the density $\rho$, and the source covariance $\mathcal{S}$, 
are gathered into the vector of unknowns $\bm = (c, \rho, \mathcal{S})$. Given expected values of cross-correlations, the objective is to find 
$\bm$ such that simulation can best reproduce the data, in the spirit 
of Full Waveform Inversion for seismic inversion, \cite{Virieux2009}. 
The reconstruction is thus formulated as a minimization problem, with an
iterative procedure in a Newton-type algorithm. Due to the large-scale computations, 
only first-order information (the gradient) is used, typically with methods such
as gradient descent, nonlinear conjugate gradients or L-BFGS~\cite{nocedal2006}. 
To avoid having to form explicitly the Jacobian matrix of the forward operator, 
the gradient is computed using the adjoint-state method, 
(e.g., \cite{chavent1974,plessix2006,Chavent2010,kern2016,barucq2019,Faucher2020}) 
which formulation must be adapted to the expected cross-correlation and the HDG discretization.
In particular, contrary to the case of FWI with direct wavefield measurements, \cite{Faucher2020}, two adjoint states have to be computed when working with the cross-correlation.

% ------------------------------------------------------------------
% ------------------------------------------------------------------
\subsection{Iterative minimization}
% ------------------------------------------------------------------
The inverse problem is formulated as a minimization of a misfit functional
$J$, defined as a distance function between observations and simulations.
We denote by $d^{[i]}_{\mathrm{obs},\bx_1,\bx_2}$ the observation
of the expected cross-correlation between the pressure field
(following \cref{section:equations})
and the field '$i$' at positions $\bx_2$ and $\bx_1$ respectively.
Similarly, $\widehat{\mathcal{C}}^{[i]}_{\bx_2, \bm}(\bx_1)$ denotes
the corresponding simulation using model parameter $\bm$.
In this work we consider the least-squares distance and the misfit functional $J$ is written as 
\begin{equation} \label{equation:misfit}
    J(\bm) = \frac{1}{2} \sum_{\bx_1 \in \rcvs} \sum_{\bx_2 \in \srcs}
    \sum_{i \in \comp_{\mathrm{d}}}
    \left| \widehat{\mathcal{C}}^{[i]}_{\bx_2, \bm}(\bx_1)
    - 
  % d^{[i]}_{\mathrm{obs},\bx_1,\bx_2} \right|^2,
    d^{[i]}_{\bx_2}(\bx_1) \right|^2,
\end{equation}
where $\comp_{\mathrm{d}}$ denotes the set of components used,
$\comp_{\mathrm{d}} \subseteq \comp$. As discussed
in \cref{section:twosteps}, one can consider a different set of positions
for $\bx_2 \in \srcs$ and $\bx_1 \in \rcvs$ in order to reduce the computational
cost, as only the cardinality of $\srcs$ determines the number of forward problems
to solve.
The simulated fields $\widehat{\mathcal{C}}^{[i]}_{\bx_2,\bm}$ are obtained from the
restriction of the total field $\brmC_{\bx_2,\bm}$ that contains all components,
and we introduce the (linear) observation operator $\mathfrak{R}$ such that,
\begin{equation}
    \mathfrak{R} \, \brmC_{\bx_2, \bm} \,=\,
    \left( \widehat{\mathcal{C}}^{[i]}_{\bx_2, \bm}(\bx_1) \right)_{\substack{\bx_1 \in \rcvs \\ i \in \comp_\mathrm{d}}} \,,
    \quad \text{with} \quad
    \brmC_{\bx_2,\bm}\,=\, \left( \widehat{\mathcal{C}}^{[i]}_{\bx_2,\bm} \right)_{i \in \comp}.
\end{equation}
where $\brmC_{\bx_2, \bm}$ is obtained from
\cref{algorithm:forward-cc-hdg}.
Using $\mathbf{d}_{\mathrm{obs}, \bx_2}$ to denote the vector of observed components,
\begin{equation}
%    \mathbf{d}_{\mathrm{obs}, \bx_2} \,=\,
%    \left( d^{[i]}_{\mathrm{obs},\bx_1,\bx_2} \right)_{\substack{\bx_1 \in \rcvs \\ i \in \comp_{\mathrm{d}}}},
    \mathbf{d}_{\bx_2} \,=\,
    \left( d^{[i]}_{\bx_2} (\bx_1) 
    \right)_{\substack{\bx_1 \in \rcvs \\ i \in \comp_{\mathrm{d}}}},
\end{equation}
the misfit functional \cref{equation:misfit}
is written in compact form as
\begin{equation}
    J(\bm) = \frac{1}{2} \sum_{\bx_2 \in \srcs}
    \left\| \mathfrak{R}\, \brmC_{\bx_2,\bm} - \mathbf{d}_{\bx_2} \right\|^2,
\end{equation}
where $\|\cdot\|$ denotes the Euclidean norm.
% ================================================
\paragraph{Representation of the model parameters}
% ================================================

In our implementation, the model parameters $(c, \rho, \mathcal{S})$ 
are represented with Lagrange basis functions on each cell, 
independently of the discretization order of the equations.
We denote by $\mathfrak{c}_e$, $\uprho_e$ and $\mathbf{S}_e$ the vectors 
of degrees of freedom representing respectively 
$c|_{K_e}$, $\uprho|_{K_e}$ and $\mathcal{S}|_{K_e}$, and by 
$\hat{\mathfrak{c}}_e$, $\hat{\uprho}_e$, $\hat{\mathbf{S}}_e$ the 
corresponding functions, defined on each element $K_e \in \mathcal{T}_h$.
The number of degrees of freedom used for the representation of 
$c$, $\rho$ and $\mathcal{S}$ on $K_e$ are denoted respectively 
$n_e^{\mathfrak{c}}$, $n_e^{\uprho}$ and $n_e^{\mathbf{S}}$, and may 
differ from one another and from $\ndof_e$, the number of degrees of 
freedom of the approximation space.
The representations are,
\begin{equation}\begin{aligned}
& \forall \bx \in K_e, \hspace*{4em}
\mathfrak{c}_e := \left( \mathfrak{c}_{e,k} \right)_{k=1,\ldots,n_e^{\mathfrak{c}}}\,, \quad
\uprho_e       := \left( \uprho_{e,k} \right)_{k=1,\ldots,n_e^{\uprho}}\,, \quad
\mathbf{S}_e   := \left( \mathbf{S}_{e,k} \right)_{k=1,\ldots,n_e^{\mathbf{S}}}\,; \\
& 
\hat{\mathfrak{c}}_e(\bx) \,=\, \sum_{k=1}^{n_e^{\mathfrak{c}}} \mathfrak{c}_{e,k} \, \varphi_{e,k}^{\mathfrak{c}}(\bx), \quad
\hat{\uprho}_e(\bx) \,=\, \sum_{k=1}^{n_e^{\uprho}} \uprho_{e,k} \, \varphi_{e,k}^{\uprho}(\bx), \quad
\text{and} \quad \hat{\mathbf{S}}_e(\bx) \,=\, \sum_{k=1}^{n_e^{\mathbf{S}}} \mathbf{S}_{e,k} \, \varphi_{e,k}^{\mathbf{S}}(\bx), \label{eq:param_dof}
\end{aligned}\end{equation}
where $\varphi_{e,k}^{\mathfrak{c}}$, $\varphi_{e,k}^{\uprho}$ and 
$\varphi_{e,k}^{\mathbf{S}}$ are $k$-th Lagrange basis functions on $K_e$ 
associated with $c$, $\rho$, $\brmS$. 
%(that may differ from $\varphi_{e,k}$ used to represent the wave solutions).
Note that in the case where the model parameters 
are represented with piecewise constants we have $n_e^\bullet=1$; 
this is however not accurate to represent complex 
media, as highlighted in \cite[Fig. 7]{Faucher2020}.
For simplicity we ignore attenuation in which case 
the wave speed \cref{eq:wavespeed} is complex-valued. Otherwise, one
has to explicit the dependency of the complex
wave speed with respect to real-valued parameters, cf.~\cite{Faucher2023viscoacoustic}.

We introduce the global vectors $\mathfrak{c}$, 
$\uprho$ and $\mathbf{S}$, 
obtained by concatenating the 
element-wise vectors $\mathfrak{c}_e$, $\uprho_e$ and 
$\mathbf{S}_{e}$ over all elements $K_e \in \mathcal{T}_h$, 
respectively. For instance for the wave speed we have 
\begin{equation}
  \mathfrak{c} \,:=\, 
  \left( \mathfrak{c}_e \right)_{K_e \in \mathcal{T}_h}\,,
\end{equation}
and similarly for the other parameters.
The misfit functional $J$ is considered as a function of 
these discretized model parameters whose coefficients are 
the unknowns of the inversion, and we write,
\begin{equation}
\mathbf{m} := (\mathfrak{c}, \uprho, \mathbf{S}).
\end{equation}

% ======================================================================
\subsection{Gradient computation with the adjoint-state method}
% ======================================================================

The gradient of the misfit functional with respect to $\mathbf{m}$ is 
computed with the \emph{adjoint-state method}. The key is to adapt this 
standard method to work with the cross-correlation 
formulation~\cite{fichtner2016,muller2024} and the HDG 
discretization, \cite{Faucher2020}. 
In the following, to lighten the notation we consider only a single 
position $\bx_2$, noting that by linearity multiple positions 
can be handled readily. We therefore write the misfit function as follows,
omitting the dependence on $\bx_2$, 
\begin{equation}
    J(\bm) = \frac{1}{2} \left\| \mathfrak{R}\brmC_{\bm} 
  % - \mathbf{d}_{\mathrm{obs}} \right\|^2 \, .
    - \mathbf{d} \right\|^2 \, .
\end{equation}

\paragraph{Notation}
The space of local HDG solutions on an element $K_e$ is denoted $\boldsymbol{\Psi}_e$,
\begin{equation}
    \boldsymbol{\Psi}_e = \mathbb{C}^{n_{\mathrm{dof},e} \times \ncomp},
\end{equation}
and $\boldsymbol{\Psi}:=(\boldsymbol{\Psi}_e)_{K_e\in\mathcal{T}_h}$ 
denotes the global solution space. 
The space of the numerical trace defined on the mesh skeleton 
is denoted $\boldsymbol{\Lambda}$: 
\begin{equation}
    \boldsymbol{\Lambda} =\mathbb{C}^{n_{\mathrm{dof}}^\Lambda},
\end{equation}
and $\boldsymbol{\Lambda}_e$ denotes the restriction 
of $\boldsymbol{\Lambda}$ to element $K_e$, such that 
$\mathcal{R}_e \Lambda \in \boldsymbol{\Lambda}_e$.

\begin{notation} \label{def:inner_product}
We denote with $( \cdot , \cdot)_{\boldsymbol{\Psi}_e}$ and with $( \cdot, \cdot )_{\boldsymbol{\Psi}_e}$ the inner product on $\boldsymbol{\Psi}_e$ and $\boldsymbol{\Psi}$ respectively, such that:
\begin{align}
    &\forall (\brmU_e , \brmV_e) \in \boldsymbol{\Psi}_e^2, \quad 
    (\brmU_e , \brmV_e)_{\boldsymbol{\Psi}_e} 
    = \sum_{i \in \comp} \sum_{k=1}^{n_{\mathrm{dof},e}} 
      \mathrm{U}_{e,k}^{[i]} \, \overline{\mathrm{V}_{e,k}^{[i]}}, 
      \\
      &\forall (\brmU , \brmV) \in \boldsymbol{\Psi}^2, \quad 
    (\brmU, \brmV)_{\boldsymbol{\Psi}} = \sum_{K_e \in \mathcal{T}_h} 
    (\brmU_e , \brmV_e)_{\boldsymbol{\Psi}_e}.
\end{align}
where $\mathrm{U}_e^{[i]}$ and $\mathrm{V}_e^{[i]}$ denote the vector blocks 
of $\brmU_e$ and $\brmV_e$, respectively.
The inner products $(\cdot, \cdot)_{\boldsymbol{\Lambda}}$ and 
$(\cdot, \cdot)_{\boldsymbol{\Lambda}_e}$ are defined analogously, as 
standard Hermitian inner products summed over the components in $\comp$.

\end{notation}

% =================================
\subsubsection{Lagrangian function}
% =================================

We introduce the Lagrangian function $\mathcal{L}$ such that
\begin{align}
\mathcal{L}(
\bm, \brmG, \Lambda_G,
\brmC, \Lambda_C,
\boldsymbol{\gamma}_{1}, \Gamma_1,
\boldsymbol{\gamma}_{2}, \Gamma_2
)
&=
\dfrac{1}{2}
\left\|
\mathfrak{R}\brmC - \bd
\right\|^2
\nonumber + \mathcal{L}_G(
\bm, \brmG, \Lambda_G,
\boldsymbol{\gamma}_{1}, \Gamma_1
)
\nonumber \\
&\quad
+ \mathcal{L}_C(
\bm, {\brmC}, \Lambda_C,
{\brmG},
\boldsymbol{\gamma}_{2}, \Gamma_2
),
\label{eq:lagrangian_def_gen}
\end{align}
where $\mathcal{L}_G$ and $\mathcal{L}_C$ denote the constraint terms associated with the HDG discretization for the simulation 
of the Green's function \cref{eq:HDG_green}  and cross-correlation \cref{eq:HDG_cc} systems, respectively:
\begin{subequations} \label{eq:Lagragian_def_loc}
\begin{align}
\mathcal{L}_G
&=
\sum_e \left (
 \left(
    \mathbb{B}_e \brmG_e
    + \mathbb{L}_e \mathcal{R}_e \Lambda_G
\right),
\, \mathcal{R}_e \Gamma_1
\right )_{\boldsymbol{\Lambda}_e}+
\sum_e
\left (
    \mathbb{A}_e \brmG_e
    + \mathbb{D}_e \mathcal{R}_e \Lambda_G
    - \mathbb{P}_e,
\, \boldsymbol{\gamma}_{1,e}
\right )_{\boldsymbol{\Psi}_e},
\\[6pt]
\mathcal{L}_C
&=
\sum_e \left (
 \left(
    \mathbb{B}_e \brmC_e
    + \mathbb{L}_e \mathcal{R}_e \Lambda_C
\right),
\, \mathcal{R}_e\Gamma_2
\right )_{\boldsymbol{\Lambda}_e}+
\sum_e
\left (
    \mathbb{A}_e \brmC_e
    + \mathbb{D}_e \mathcal{R}_e \Lambda_C
    - \mathbb{M}_e(\overline{\brmG}_e, \bm),
\, \boldsymbol{\gamma}_{2,e}
\right )_{\boldsymbol{\Psi}_e},
\end{align}
\end{subequations}
Here $\brmG = (\brmG_e)_{K_e \in \mathcal{T}_h} \in \boldsymbol{\Psi}$ and 
$\brmC = (\brmC_e)_{K_e \in \mathcal{T}_h} \in \boldsymbol{\Psi}$, with $\brmG_e$, $\brmC_e \in \boldsymbol{\Psi}_e$, 
denote the concatenations of the coefficients of the local HDG solutions, 
and $\Lambda_G, \Lambda_C \in \boldsymbol{\Lambda}$ are the ones for the 
numerical traces.  
The Lagrange multipliers (adjoint variables) are the two couples
$(\boldsymbol{\gamma}_1,\Gamma_1)$ and $(\boldsymbol{\gamma}_2,\Gamma_2)$,
where
$\boldsymbol{\gamma}_1 = (\boldsymbol{\gamma}_{1,e})_{K_e \in \mathcal{T}_h} \in \boldsymbol{\Psi}$ and 
$\boldsymbol{\gamma}_2 = ({\gamma}_{2,e})_{K_e \in \mathcal{T}_h} \in \boldsymbol{\Psi}$ 
represent the adjoint velocity fields,
with $\boldsymbol{\gamma}_{1,e}, \boldsymbol{\gamma}_{2,e} \in \boldsymbol{\Psi}_e$, 
and $\Gamma_1$, $\Gamma_2 \in \boldsymbol{\Lambda}$ represent the adjoint pressure fields. Similarly to $\mathbf{G}$, we denote by $\gamma_k^{[i]}$ the block of $\gamma_k$ corresponding to the $i$  unknown, where $i \in \comp$.
Note that the model parameters of $\bm$ are involved in the HDG 
matrices $\mathbb{A}_e$, $\mathbb{L}_e$  and $\mathbb{M}_e$.

We recall that the right-hand side matrix $\mathbb{M}_e$ is 
defined in \cref{eq:hdg-rhs-step2}, and depends on the source-covariance 
$\mathcal{S}$. Replacing $\mathcal{S}$ by its finite-dimensional representation 
$\hat{\mathbf{S}}_e$ \cref{eq:param_dof}, we have
$\mathbb{M}_e(\overline{\brmG}_e, \mathbf{m}) = \mathbb{M}_e(\overline{\hat{\brmG}}_e, \hat{\mathbf{S}}_e) = ( \mathrm{M}_e \; \mathbf{0} )^{\mathrm{t}}$.
In particular, the $l$-th component of the non-zero block of $\mathbb{M}_e$ is given by
\begin{equation}
\left(
    \mathrm{M}_e \big( \overline{\hat{\brmG}}_e, \hat{\mathbf{S}}_e \big) 
\right)_l
=
\left\langle
    \left( \sum_{m=1}^{n_e^{\mathbf{S}}} \mathbf{S}_{e,m} \, \varphi_{e,m}^{\mathbf{S}}(\bx) \right)
    \left( \sum_{r=1}^{\ndof_e} \overline{\mathrm{G}_{e,r}^{[p]}} \, \varphi_{e,r}(\bx) \right),
    \varphi_{e,l}(\bx)       
\right\rangle_{K_e},
\end{equation}
for $l = 1, \ldots, \ndof_e$. Here, 
we recall that $\langle \cdot , \cdot \rangle_{K_e}$ 
denotes the $L^2$ Hermitian inner product on $K_e$, such
that for two functions $f$ and $g$, $\langle f, \, g\rangle_{K_e} \,=\, \int_{K_e} f \, \overline{g}$.

\begin{proposition} \label{prop:source_mat}
The following equality holds for any real-valued function $\hat{\mathbf{S}}_e$
given by \cref{eq:lagrange_rep},
\begin{equation}
\left(
\mathbb{M}_e\left(\overline{\hat{\brmG}}_e, \hat{\mathbf{S}}_e\right),
\boldsymbol{\gamma}_{2,e}
\right)_{\boldsymbol{\Psi}_e}
=
\left(
\overline{\hat{\brmG}}_e,
\mathbb{M}_e\left( \hat{\boldsymbol{\gamma}}_{2,e}, \overline{\hat{\mathbf{S}}}_e\right)
\right)_{\boldsymbol{\Psi}_e},
\end{equation}
where $\hat{\boldsymbol{\gamma}}_{2,e}$ denotes the function associated with the degrees of freedom $\boldsymbol{\gamma}_{2,e}$, expressed in the Lagrange basis $\varphi_l$, \cref{eq:lagrange_rep}.
\end{proposition}
\begin{proof}
For clarity, we omit the element-wise subscript $e$ in the following expressions. We have,
\begin{equation}\begin{aligned}
    \left( 
        \mathbb{M}\big( \overline{\hat{\brmG}}, \hat{\mathbf{S}} \big),
        \boldsymbol{\gamma}_{2}
    \right) 
    &= \left( 
        \mathrm{M}\big( \overline{\hat{\mathcal{G}^{[p]}}}, \hat{\mathbf{S}} \big),
        \gamma^{[p]}_{2}
    \right) 
    =
    \sum_{l=1}^{\ndof_e}
    \left\langle
        \left( \sum_{m=1}^{n_e^{\mathbf{S}}} \mathbf{S}_m \varphi_m^{\mathbf{S}} \right)
        \left( \sum_{r=1}^{\ndof_e} \overline{\mathrm{G}^{[p]}_r} \, \varphi_r \right),
        \varphi_l
    \right\rangle
    \overline{\gamma}_{2,l}^{[p]} \\
    &=
    \sum_{l=1}^{\ndof_e}
    \sum_{m=1}^{n_e^{\mathbf{S}}}
    \sum_{r=1}^{\ndof_e} 
    \mathbf{S}_m \, \overline{\mathrm{G}^{[p]}_r}
    \left\langle
        \varphi_m^{\mathbf{S}} \varphi_r, \varphi_l
    \right\rangle
    \overline{\gamma}_{2,l}^{[p]} \\
    &=
    \sum_{r=1}^{\ndof_e}
    \overline{\mathrm{G}^{[p]}_r}
    \left\langle
        \varphi_r,
       \left( \sum_{m=1}^{n_e^{\mathbf{S}}} \overline{\mathbf{S}_m} \varphi_m^{\mathbf{S}} \right)
        \left( \sum_{l=1}^{\ndof_e} \overline{\gamma}_{2,l}^{[p]} \varphi_l \right)
    \right\rangle \\
    &=
    \left( 
        \overline{\hat{\brmG}}, \mathbb{M}( \hat{\boldsymbol{\gamma}}_2, \overline{\hat{\mathbf{S}}} )
    \right).
\end{aligned}\end{equation}

\end{proof}
% ===============================================
\subsubsection{Selection of the adjoint states}
% ===============================================

Let us denote by $(\tilde{\brmG}, \tilde{\Lambda}_G)$ the exact 
HDG solutions for the  Green's function problem~\cref{eq:HDG_green}, 
and by $(\tilde{\brmC}, \tilde{\Lambda}_C)$ 
the HDG solutions for the cross-correlation problem~\cref{eq:HDG_cc}, both 
computed with some model parameters $\tilde{\mathbf{m}}$. 
By definition, we have that, for any choice of Lagrange 
multipliers $\boldsymbol{\gamma}_1, \Gamma_1, \boldsymbol{\gamma}_2, \Gamma_2$,
\begin{equation}
  \mathcal{L}(\tilde{\bm}, \tilde{\brmG}, \tilde{\Lambda}_G, \tilde{\brmC}, \tilde{\Lambda}_C,
\boldsymbol{\gamma}_{1}, \Gamma_1, \boldsymbol{\gamma}_{2}, \Gamma_2)
= J(\tilde{\bm}).
\end{equation}
% From now on, every partial derivative of $\mathcal{L}$ with respect to 
% $\brmG$, $\Lambda_G$, $\brmC$, or $\Lambda_C$ is evaluated at 
% $\tilde{\brmG}$, $\tilde{\Lambda}_G$, $\tilde{\brmC}$, or $\tilde{\Lambda}_C$, 
% respectively.

Since the Lagrangian involves complex-valued quantities, 
we use Wirtinger calculus formalism, cf.~\cite{brandwood1983,koor2023,
kreutz2009,barucq2019}, which provides a systematic framework 
for differentiating real-valued functions of complex variables
for non-holomorphic quantity. This consists 
in treating $\brmG$ and $\overline{\brmG}$, as well as $\brmC$ and 
$\overline{\brmC}$, as independent variables. 
% Since $\mathcal{L}(\tilde{\bm}, \tilde{\brmG}, \tilde{\Lambda}_G, \tilde{\brmC}, 
% \tilde{\Lambda}_C) = J$
% for any choice of Lagrange multipliers, 
The gradient of $J$ with respect to $\bm$ follows from the chain rule as
\begin{equation}
\begin{aligned}
D_{\mathbf{m}}J
=
\Re \Big( \, &
\partial_{\mathbf{m}} \mathcal{L} 
\,+\, \partial_{\mathbf{G}}\mathcal{L} \, D_{\mathbf{m}}\mathbf{G} 
\,+\, \partial_{\overline{\mathbf{G}}}\mathcal{L} \, D_{\mathbf{m}}\overline{\mathbf{G}} 
\,+\, \partial_{\mathbf{C}}\mathcal{L} \, D_{\mathbf{m}}\mathbf{C} \\
&
\,+\, \partial_{\overline{\mathbf{C}}}\mathcal{L} \, D_{\mathbf{m}}\overline{\mathbf{C}} 
\,+\, \partial_{\Lambda_G}\mathcal{L} \, D_{\mathbf{m}}\Lambda_G
\,+\, \partial_{\Lambda_C}\mathcal{L} \, D_{\mathbf{m}}\Lambda_C
\, \Big),
\end{aligned}
\label{eq:lagrangian_deriv_wirtinger}
\end{equation}
where all functions are evaluated at
$(\tilde{\bm}, \tilde{\brmG}, \tilde{\Lambda}_G, \tilde{\brmC},
\tilde{\Lambda}_C)$, $D_\bm$ denotes the total differential with respect to $\bm$,
and $\partial_{\bullet}$ denotes the partial derivatives.
Note that the partial derivatives with respect 
to $\overline{\Lambda}_G$ and $\overline{\Lambda}_C$ vanish since, 
$\mathcal{L}$ does not depend on the conjugate of 
$\Lambda_G$ and $\Lambda_C$ in \cref{eq:Lagragian_def_loc}.

The adjoint-states 
% $(\gamma_1, \Gamma_1, \gamma_2, \Gamma_2)$ 
$\tilde{\boldsymbol{\gamma}}_1$, $\tilde{\Gamma}_1$, $\tilde{\boldsymbol{\gamma}}_2$, 
$\tilde{\Gamma}_2$
are selected such that the terms involving $D_{\bm}\tilde{\brmG}$, $D_{\bm}\tilde{\brmC}$, 
$D_{\bm}\tilde{\Lambda}_G$, and $D_{\bm}\tilde{\Lambda}_C$ in 
\cref{eq:lagrangian_deriv_wirtinger} vanish, in particular:
\begin{equation}
 \frac{\partial \mathcal{L}}{\partial \mathbf{G}} +
\frac{\partial \mathcal{L}}{\partial \overline{\mathbf{G}}}  = 0, \quad
\frac{\partial \mathcal{L}}{\partial \Lambda_G} = 0, \quad 
 \frac{\partial \mathcal{L}}{\partial \mathbf{C}} +
\frac{\partial \mathcal{L}}{\partial \overline{\mathbf{C}}} = 0, \quad
\frac{\partial \mathcal{L}}{\partial \Lambda_C} = 0.
\end{equation}

Computing the partial derivatives 
for every direction $\delta_{C,e} \in \boldsymbol{\Psi}_e$ and $\delta_\Lambda \in \boldsymbol{\Lambda}$:
\begin{subequations}
\begin{empheq}[left=\empheqlbrace]{align}
\dfrac{\partial \mathcal{L}}{\partial \brmC_e} \, \delta_{C,e} 
+ \dfrac{\partial \mathcal{L}}{\partial \overline{\brmC}_e}\delta_{C,e}
&= 
\frac{1}{2}\left(\mathfrak{R}_e\delta_{C,e}, \mathfrak{R}_e\tilde{\brmC}_e - \bd\right)
+
\frac{1}{2}\overline{\left(\mathfrak{R}_e\delta_{C,e}, \mathfrak{R}_e\tilde{\brmC}_e - \bd\right)}
\nonumber \\
&\quad +
\left(\mathbb{B}_e \delta_{C, e}, \mathcal{R}_e \Gamma_2\right)
+
\left(\mathbb{A}_e \delta_{C, e}, \boldsymbol{\gamma}_{2,e}\right) 
\nonumber \\
&= 
\Re \left( \mathfrak{R}_e\delta_{C,e}, \mathfrak{R}_e\tilde{\brmC}_e - \bd \right)   
+
\left(\mathbb{B}_e \delta_{C, e}, \mathcal{R}_e \Gamma_2\right)
+
\left(\mathbb{A}_e \delta_{C, e}, \boldsymbol{\gamma}_{2,e}\right),
\\[6pt]
\dfrac{\partial \mathcal{L}}{\partial \Lambda_C} \, \delta_\Lambda
&= \sum_e
\left( 
     \mathbb{L}_e \mathcal{R}_e \delta_\Lambda,
    \mathcal{R}_e\Gamma_2
\right)
+
\sum_e
\left(
    \mathbb{D}_e \mathcal{R}_e \delta_\Lambda,
    \boldsymbol{\gamma}_{2,e}
\right), 
\end{empheq}
\label{eq:source_adjoint}
\end{subequations}
where $\partial \mathcal{L} / \partial \brmC_e$ denotes the partial derivative 
with respect to the local component $\brmC_e$ on element $K_e$, treating all 
other $\brmC_{e'}$, $e' \neq e$, as fixed.
These conditions are satisfied for $(\tilde{\boldsymbol{\gamma}}_{2,e},\tilde{\Gamma}_2)$ solutions to
\begin{subequations}\label{eq:HDG_adjoint_C}
\begin{empheq}[left=\empheqlbrace]{align}
    \mathbb{A}_e^* \tilde{\boldsymbol{\gamma}_{2,e}}
    + \mathbb{B}_e^* \mathcal{R}_e \tilde{\Gamma_2}
&=
- \mathfrak{R}^*_e
\left(
    \mathfrak{R}_e \tilde{\brmC_e} - \bd
\right), \qquad \forall K_e \in \mathcal{T}_h,
\label{eq:source_adjoint_system_a}
\\[6pt]
\sum_e 
\left(
    \mathcal{R}_e^\mathrm{t} \mathbb{D}_e^* \tilde{\boldsymbol{\gamma}_{2,e}}
    + \mathcal{R}_e^\mathrm{t}\mathbb{L}_e^* \mathcal{R}_e \tilde{\Gamma_2}
\right)
&=
0.
\label{eq:source_adjoint_system_b}
\end{empheq}
\end{subequations}
The real part operator from the cost functional is removed since the condition must hold for arbitrary complex variations.
We recognize an adjoint state similar to the FWI problem, 
see \cite[Eq. (52)]{Faucher2020}, in which the residuals are used 
as right-hand sides.

For the derivatives $\partial_{\brmG_e}\mathcal{L}$ and 
$\partial_{\Lambda_G}\mathcal{L}$ in directions 
$\delta_{G,e}$ and $\delta_{\Lambda}$, we have,
\begin{equation}
    \begin{array}{@{\,}l@{\quad}l@{\quad}l@{\,}l}
        \dfrac{\partial \mathcal{L}_{G} }{\partial \brmG_e} \delta_{G,e}
        &= \displaystyle\sum_e  ( \mathbb{A}_e \delta_{G,e} , \boldsymbol{\gamma}_1  ) 
        + \displaystyle\sum_e (  \mathbb{B}_e \delta_{G,e}, \mathcal{R}_e \Gamma_1  ), 
        & \dfrac{\partial \mathcal{L}_G}{\partial\overline{\brmG}} \delta_{G,e} = 0, & 
        \qquad \forall K_e \in \mathcal{T}_h,  \\[15pt]
        \dfrac{\partial \mathcal{L}_G }{\partial \Lambda_G} \delta_\Lambda 
        &= \displaystyle\sum_e  ( \mathbb{D}_e \mathcal{R}_e \delta_\Lambda, \boldsymbol{\gamma}_1  ) 
        + \displaystyle\sum_e  (  \mathbb{L}_e \mathcal{R}_e \delta_\Lambda, \mathcal{R}_e \Gamma_1  ), 
        &  \dfrac{\partial \mathcal{L}_C }{\partial \Lambda_G} \delta_\Lambda = 0,
    \end{array}
\end{equation}
The conjugate of $\mathbf{G}_e$ appear linearly in the RHS matrix $\mathbb{M}_e$, therefore, using \cref{prop:source_mat}:
\begin{equation}
\dfrac{\partial \mathcal{L}_C}{\partial \overline{\mathbf{G}}_e} \, \delta_{G,e} 
= \left ( \mathbb{M}_e(\delta_{G,e}, \mathbf{S}_e) , \boldsymbol{\gamma}_{2, e} \right )= \left ( \delta_{G,e} , \mathbb{M}_e( \boldsymbol{\gamma}_{2, e} , \overline{\brmS}_e ) \right ), 
\qquad 
\dfrac{\partial \mathcal{L}_C}{\partial \mathbf{G}_e} \, \delta_{G,e} = 0.
\end{equation}
The second adjoint state, 
$(\tilde{\boldsymbol{\gamma}}_{1,e}, \tilde{\Gamma}_1)$ is thus selected as 
solution to:
\begin{subequations}\label{eq:HDG_adjoin_model}
    \begin{empheq}[left=\empheqlbrace]{align}
     \mathbb{A}_e^* \tilde{\boldsymbol{\gamma}}_{1, e} + \mathbb{B}_{e}^* \mathcal{R}_e \tilde{\Gamma}_{1} &= \mathbb{M}_e( \tilde{\boldsymbol{\gamma}}_{2, e} , \overline{\brmS}_e), & \forall K_e \in \mathcal{T}_h. \label{eq:HDG_adjoin_model_loc}
      \\
    \sum_e \left ( \mathcal{R}_e^\mathrm{t} \mathbb{D}_e ^*\tilde{\boldsymbol{\gamma}}_{1, e} + \mathcal{R}_e^\mathrm{t} \mathbb{L}_e^* \mathcal{R}_e \tilde{\Gamma}_{1} \right )  &= 0,    
   \end{empheq}
\end{subequations}

Both adjoint-state systems \cref{eq:HDG_adjoint_C,eq:HDG_adjoin_model} have 
the same structure as the forward HDG problem \cref{eq:HDG_mat_full}, however 
with complex conjugation of the local matrices. 
Nevertheless, they result in the global system in terms of 
the adjoint of the forward global matrix, 
$\mathbb{K}^*$, cf.~\cite{Faucher2020}. Specifically, 
the adjoint states under the HDG discretization are obtained as follows.
\begin{itemize}
    \item The first adjoint state 
    $(\tilde{\boldsymbol{\gamma}}_{2,e},\widetilde{\Gamma}_2)$ is obtained 
    by solving the global problem,
    \begin{equation}
    \mathbb{K}^* \tilde{\Gamma}_2 = \mathbb{H}_2,
    \qquad
    \mathbb{H}_2 = -\sum_e \mathcal{R}_e^\mathrm{t} 
    \left( \mathbb{D}_e^* \mathbb{A}_e^{-*} \mathfrak{R}_e^*
    \left( \mathfrak{R}_e \tilde{\brmC}_e - \bd \right) \right),
    \end{equation}
and then $\tilde{\boldsymbol{\gamma}}_{2,e}$ is given from 
$\tilde{\Gamma}_2$ via \cref{eq:source_adjoint_system_a}.
    \item The second adjoint state is obtained by solving the global problem
    \begin{equation}
    \mathbb{K}^* \tilde{\Gamma}_1 = \mathbb{H}_1,
    \qquad
    \mathbb{H}_1 = -\sum_e \mathcal{R}_e^\mathrm{t} 
    \left( \mathbb{D}_e^* \mathbb{A}_e^{-*} 
    \mathbb{M}_e( \tilde{\boldsymbol{\gamma}}_{2,e}, \overline{\brmS}_e) \right),
    \end{equation}
and the local solution $\tilde{\boldsymbol{\gamma}}_{1,e}$ 
is obtained from $\tilde{\Gamma}_1$ using \cref{eq:HDG_adjoin_model_loc}.
\end{itemize}

% ===============================================
\subsubsection{Expression of the gradients}
% ===============================================

After computing the HDG solutions $(\tilde{\brmG}, \tilde{\Lambda}_G)$ 
of~\cref{eq:HDG_green}, and $(\tilde{\brmC}, \tilde{\Lambda}_C)$ 
for~\cref{eq:HDG_cc}, the two adjoint states ($\tilde{\boldsymbol{\gamma}}_1$, $\tilde{\Gamma}_1$)
and $(\tilde{\boldsymbol{\gamma}}_2$, $\tilde{\Gamma}_2)$ are obtained with 
\cref{eq:HDG_adjoint_C,eq:HDG_adjoin_model}.
Then, the gradient of $J$ in \cref{eq:lagrangian_deriv_wirtinger} 
reduces to the partial derivative of the Lagrangian with respect to $\bm$:
\begin{equation}
    D_{\bm} J = \Re \left( \partial_{\bm} \mathcal{L} \right).
\end{equation}
The precise derivation depends on the model parameter inverted,
i.e., the source covariance term $\mathbf{S}$ 
and the background parameters  $(\mathfrak{c}, \uprho)$. 
% as they enter the Lagrangian through different terms, $\mathbf{S}$ appears only in $\mathbb{M}_e$, while 
% $\mathfrak{c}$ and $\uprho$ appear in $\mathbb{A}_e$ and $\mathbb{L}_e$.
Namely, $\mathbf{S}$ only appears in the term $\mathbb{M}_e$, while 
$\mathfrak{c}$ and $\uprho$ appear in the local matrices 
$\mathbb{A}_e$ and $\mathbb{L}_e$. 
For any model parameter $m\in \bm$, the gradient writes as
\begin{equation}\label{equation:gradient-full}
\begin{aligned}
    D_{m} J \cdot \delta_{m}
  & \,=\, \Re \left( 
        \partial_{m} \mathcal{L}_G \cdot \delta_{m} 
        + \partial_{m} \mathcal{L}_C \cdot \delta_{m} 
      \right) \\
  & \,=\, \Re \Bigg( \sum_{K_e \in \mathcal{T}_h} 
    \left( 
    \partial_{m} \mathbb{L}_e \cdot \delta_{m} \, 
    \mathcal{R}_e \tilde{\Lambda}_G,
    \mathcal{R}_e \tilde{\Gamma}_1
    \right)_{\boldsymbol{\Lambda}_e}
    +\,
    \left( 
    \partial_{m} \mathbb{A}_e \cdot \delta_{m} \, \tilde{\brmG}_e,
    \tilde{\boldsymbol{\gamma}}_{1,e}
    \right)_{\boldsymbol{\Psi}_e} \\
  & \phantom{\,=\, \sum_{K_e \in \mathcal{T}_h}} +
    \left( 
    \partial_{m} \mathbb{L}_e \cdot \delta_{m} \, 
    \mathcal{R}_e \tilde{\Lambda}_C,
    \mathcal{R}_e \tilde{\Gamma}_2
    \right)_{\boldsymbol{\Lambda}_e}
    +\,
    \left( 
    \partial_{m} \mathbb{A}_e \cdot \delta_{m} \, \tilde{\brmC}_e,
    \tilde{\boldsymbol{\gamma}}_{2,e}
    \right)_{\boldsymbol{\Psi}_e} \\
  & \phantom{\,=\, \sum_{K_e \in \mathcal{T}_h}}
    -\,
    \left( 
        \mathbb{M}_e\left(\overline{\tilde{\brmG}}_e, \delta_{m}\right), 
        \tilde{\boldsymbol{\gamma}}_{2,e} 
   \right)_{\boldsymbol{\Psi}_e} \Bigg)\,.
\end{aligned} \end{equation}

\paragraph{Derivative for $\mathbf{S}$}

We first consider the derivative with respect to the
source covariance. In \cref{equation:gradient-full}, only the last term
depends on the source covariance via $\mathbb{M}_e$, hence we have,
\begin{equation}
    D_{\brmS}J \cdot \delta
    = \Re \left( \partial_{\brmS} \mathcal{L} \cdot \delta \right)
    = -\Re \left( \sum_{K_e \in \mathcal{T}_h} \left(
        \mathbb{M}_e\left(\overline{\hat{\tilde{\brmG}}}_e, \hat{\delta_{e}}\right),
        \tilde{\boldsymbol{\gamma}}_{2,e}
      \right)_{\boldsymbol{\Psi}_e} \right),
\end{equation}
where the direction $\delta$ is an arbitrary perturbation of the source covariance
$\brmS$, and $\delta_e$ denotes the vector of coefficients of this perturbation
restricted to the element $K_e$, with $\hat{\delta_e}$ the corresponding function
in the local approximation space(in the same way that
$\hat{\tilde{\brmG}}_e$ denotes the function associated
with the coefficients $\tilde{\brmG}_e$).

Therefore, to compute the gradient with respect to
the source covariance, it is only necessary to have 
the adjoint state $(\tilde{\boldsymbol{\gamma}}_{2}, \tilde{\Gamma}_2)$.

\paragraph{Derivative for $\mathfrak{c}$ and $\uprho$}

For the derivatives with respect to the wave speed $\mathfrak{c}$
and density $\uprho$, the first four terms in \cref{equation:gradient-full} remains. Here, both adjoint states are required, and one further
has to derive the local HDG matrices $\mathbb{A}$ and $\mathbb{L}$, similarly
to what is done in the traditional adjoint-state method for FWI, \cite{Faucher2020}. 
For the wave speed $\mathfrak{c}$, using the 
direction $\delta_{\mathfrak{c}} = (\delta_{\mathfrak{c}_e})_{K_e \in \mathcal{T}_h}$,
we have
\begin{equation}
    D_{\mathfrak{c}} J \cdot \delta_{\mathfrak{c}} 
    = \Re \left( 
        \partial_{\mathfrak{c}} \mathcal{L}_G \cdot \delta_{\mathfrak{c}} 
        + \partial_{\mathfrak{c}} \mathcal{L}_C \cdot \delta_{\mathfrak{c}} 
      \right),
\end{equation}
with
\begin{subequations}
\begin{align}
\partial_{\mathfrak{c}} \mathcal{L}_G \cdot \delta_{\mathfrak{c}}
&= \sum_{K_e \in \mathcal{T}_h}
\left(
\partial_{\mathfrak{c}_e} \mathbb{L}_e \cdot \delta_{\mathfrak{c}_e} \, 
\mathcal{R}_e \tilde{\Lambda}_G,
\mathcal{R}_e \tilde{\Gamma}_1
\right)_{\boldsymbol{\Lambda}_e}
+
\left(
\partial_{\mathfrak{c}_e} \mathbb{A}_e \cdot \delta_{\mathfrak{c}_e} \, 
\tilde{\brmG}_e,
\tilde{\boldsymbol{\gamma}}_{1,e}
\right)_{\boldsymbol{\Psi}_e},
\\
\partial_{\mathfrak{c}} \mathcal{L}_C \cdot \delta_{\mathfrak{c}}
&= \sum_{K_e \in \mathcal{T}_h}
\left( 
\partial_{\mathfrak{c}_e} \mathbb{L}_e \cdot \delta_{\mathfrak{c}_e} \, 
\mathcal{R}_e \tilde{\Lambda}_C,
\mathcal{R}_e \tilde{\Gamma}_2
\right)_{\boldsymbol{\Lambda}_e}
+
\left( 
\partial_{\mathfrak{c}_e} \mathbb{A}_e \cdot \delta_{\mathfrak{c}_e} \, 
\tilde{\brmC}_e,
\tilde{\boldsymbol{\gamma}}_{2,e}
\right)_{\boldsymbol{\Psi}_e},
\end{align}
\end{subequations}
and similarly for $\uprho$. The partial derivatives of the 
HDG matrices $\mathbb{A}_e$ and $\mathbb{L}_e$ are given in 
\cite{Faucher2020}.
We note that the gradient with respect to $\mathbf{S}$ only requires the 
adjoint state $(\tilde{\boldsymbol{\gamma}}_2, \tilde{\Gamma}_2)$. Since $\mathbf{S}$ does 
not appear in the Green's function wave problem \cref{eq:HDG_green}, the latter does not 
need to be recomputed when $\mathbf{S}$ is updated. Furthermore, the cross-correlation depends 
linearly on $\mathbf{S}$, implying that its reconstruction corresponds to a linear inverse 
problem for which alternative strategies can be employed.

\subsection{Minimization strategy}
\label{subsection:minimization-strategy}

The iterative minimization follows a quasi-Newton scheme, where
the model parameters are updated at each iteration $k$ according to
\begin{equation}
    \bm_{k+1} = \bm_{k} - \alpha_k \,\,\mathbf{s}_k \, ,
\end{equation}
where $\mathbf{s}_k$ denotes the search direction, computed from the
gradient, and $\alpha_k$ is the step length, \cite{nocedal2006}.
Full Newton (and Gauss-Newton) methods require the explicit formation
and regularization of the Hessian, which is generally computationally
prohibitive for large-scale applications \cite{Virieux2009,metivier2017,faucher2017}. 
Instead, we employ a quasi-Newton algorithm that avoids explicit Hessian formation while using only gradient information to construct the search direction.
In our work, the search direction is obtained with the 
L-BFGS algorithm, which constructs 
an approximation of the inverse Hessian through a 
two-loop recursion using (first-order) information from previous 
iterations \cite{nocedal2006}. The step length is 
computed using a linesearch algorithm \cite{Chavent2010,faucher2017}.

The method can be applied to both the background model parameters
$(\mathfrak{c},\,\uprho)$ and the source covariance $\mathbf{S}$.
Nevertheless, as mentioned above, the inversion with respect to
the source covariance is linear. When only $\mathbf{S}$ is
updated (i.e., $\mathfrak{c}$ and $\uprho$ are kept fixed), the
Green's function $\tilde{\brmG}$ remains unchanged with the
iterations. Consequently, $J$ reduces to a quadratic functional
with a constant Hessian, allowing the use of dedicated approaches
such as the conjugate gradient method \cite{nocedal2006}.

% --------------------------------------------------------------------------
\section{Numerical experiments}
\label{section:numerical}
% --------------------------------------------------------------------------
\newcommand{\misfitfwi}{J_{\mathrm{fwi}}}
\newcommand{\misfitcc} {J_{\mathrm{cc}}}
In this section, we perform numerical experiments of
nonlinear inversion using synthetic data, comparing
the use of expected value of the cross-correlations
with that of the direct wavefield. The first approach
corresponds to the traditional full waveform inversion
(FWI), which relies on controlled sources to acquire
the data, e.g., \cite{Virieux2009,faucher2017}. The
second approach, that we refer to as Full Cross-Correlation
Waveform Inversion (FCCWI), instead relies on the
expected cross-correlations of ambient signals,
following the formulation introduced in
\cref{section:inversion}.
In both cases, we follow the iterative minimization
strategy described in \cref{subsection:minimization-strategy}, 
while
working with the respective data type. The associated
misfit functionals are denoted by $\misfitfwi$ and
$\misfitcc$ for FWI and FCCWI, respectively. We focus
on the pressure fields, such that we have,
\begin{equation}\label{equation:misfits}
\begin{aligned}
  \misfitfwi(c,\rho)   & \,=\,
  \dfrac{1}{2}\sum_{\bx_1 \in \rcvs} \sum_{\bx_2 \in \srcs} \,
    \left| \mathcal{G}^{p}_{\bx_2}(\bx_1) -
   % d_{\mathrm{obs}}^{\mathrm{fwi}}(\bx_1, \bx_2) \right|^2, \\
   d_{\bx_2}^{\mathrm{G}}(\bx_1) \right|^2, \\
  \misfitcc(c,\rho,\mathcal{S}) & \,=\,
  \dfrac{1}{2}\sum_{\bx_1 \in \rcvs} \sum_{\bx_2 \in \srcs} \,
    \left| \mathcal{C}^{p}_{\bx_2}(\bx_1) -
    % d_{\mathrm{obs}}^{\mathrm{CC}}(\bx_1, \bx_2) \right|^2, \\
    d_{\bx_2}^{\mathrm{C}}(\bx_1) \right|^2, 
\end{aligned}\end{equation}
Here, $\mathcal{G}^{p}_{\bx_2}(\bx_1)$ denotes the pressure component of the 
simulated Green's function at receiver position $\bx_1$ for a source located 
at $\bx_2$, obtained from the resolution of \cref{eq:HDG_green}, and 
% $d_{\mathrm{obs}}^{\mathrm{fwi}}(\bx_1, \bx_2)$ 
$d_{\bx_2}^{\mathrm{G}}(\bx_1)$
denotes the corresponding 
observed data used in the classical FWI setting. Similarly, 
$\mathcal{C}^{p}_{\bx_2}(\bx_1)$ denotes the pressure component of the 
simulated cross-correlation expected value, obtained from \cref{eq:HDG_cc}, and 
% $d_{\mathrm{obs}}^{\mathrm{CC}}(\bx_1, \bx_2)$ 
$d_{\bx_2}^{\mathrm{C}}(\bx_1)$
the corresponding observed 
cross-correlation data. The misfit functional $\misfitfwi$ therefore 
corresponds to the standard FWI formulation based on direct wavefield 
measurements, while $\misfitcc$ corresponds to the cross-correlation-based 
formulation introduced in \cref{equation:misfit}, restricted here to the 
pressure component for simplicity. 
We carry out three synthetic experiments:
\begin{enumerate}
\item In \cref{subsection:toy-problem}, we 
      examine a toy problem with data available 
      on all boundaries, using a background 
      medium containing distinct objects with 
      contrasting properties.

\item In \cref{subsection:ot}, we consider a 
      seismic configuration with data available 
      only near the surface, based on the 
      Overthrust background model of \cite{Aminzadeh1994seg}.

\item In \cref{subsection:3d}, we demonstrate 
      the feasibility of 3D reconstructions and 
      investigate the associated computational cost.
\end{enumerate}

In all cases, the minimization problem is solved using the L-BFGS algorithm 
\cite{nocedal2006}, following the adjoint-state gradient computation 
described in \cref{section:inversion}.

% -------------------------------
\subsection{Experiment 1: Reconstruction of contrasting objects with full data}
\label{subsection:toy-problem}
% -------------------------------

We first consider an experiment with two objects 
(a square and a disk) embedded in a constant 
background medium. Each object has a different 
wave speed, as shown in \cref{figure:toy-true_cp}. 
The domain is taken to be the unit square, with the 
wave speed varying from \num{1} in the background 
to \num{2} within the square object. The density is 
set to a constant value of $\rho = 1$ throughout the 
entire domain. The source covariance $\mathcal{S}$ 
is considered with a smooth pattern, pictured see \cref{figure:toy-true_sc}.

% ============================
\begin{figure}[ht!] \centering
  \subfloat[Wave speed model.]
  {\includegraphics[scale=1]{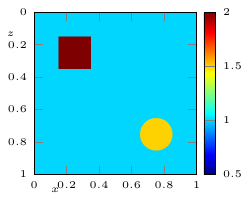}
   \label{figure:toy-true_cp}} \hspace*{3em}
  \subfloat[Source covariance model.]
  {\includegraphics[scale=1]{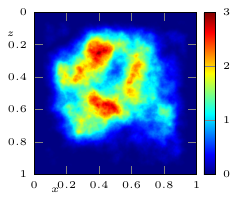}
   \label{figure:toy-true_sc}}
  \caption{Experiment on the unit-square with the wave speed
           composed of two inclusions. The density is taken
           constant. 
           The data for inversion are acquired near the 
           four boundaries.}
  \label{figure:toy-true}
\end{figure}

For the synthetic data set, absorbing boundary conditions are 
imposed on all four sides of the domain. A total of \num{149} 
receivers are considered ($\rcvs$ in \cref{equation:misfits}) 
and are distributed along the boundaries. For the set $\srcs$, 
we select a subset of \num{16} positions, uniformly 
distributed around the domain. We also add \num{10}\si{\decibel} 
noise to the synthetic data to avoid inverse crime. 
The reconstruction starts from a constant wave speed 
background, and we use integer frequencies from \num{2} to 
\num{10}\si{\Hz}. The iterative minimization of the misfit 
functionals in \cref{equation:misfits} is performed with 
30 iterations per frequency. The results obtained using 
$\misfitfwi$ and $\misfitcc$ are shown in 
\cref{figure:toy-reconstruction-cp}. For the 
cross-correlation-based inversion, the source covariance is assumed to be unknown 
and remains fixed to a Gaussian profile.

\begin{figure}[ht!] \centering
  \subfloat[FWI reconstruction using
            $\misfitfwi$.]
  {\makebox[15em]{\includegraphics[scale=1]{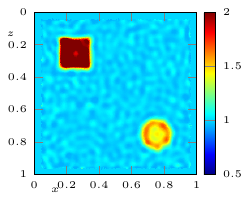}}}
  \hspace*{3em}
  \subfloat[FCCWI reconstruction using $\misfitcc$.]
  {\makebox[15em]{\includegraphics[scale=1]{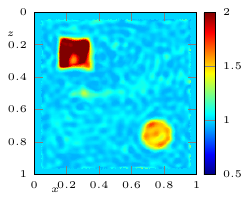}}
   \label{figure:toy-reconstruction-cp_cc}}
  \caption{Wave speed reconstructions 
         using either the direct wavefield data (FWI) 
         or using cross-correlation data (FCCWI). 
         Inversion starts from a constant background.
         The source covariance is not inverted and remains set
         to an initial guess which consists of a Gaussian
         profile.}
  \label{figure:toy-reconstruction-cp}
\end{figure}

We observe that the two reconstruction setups yield relatively 
similar results. The positions and contrasts of the two objects 
are well recovered, although the FWI approach appears to be 
slightly more accurate, for instance in the lower part of the 
square object. We also notice the presence of background 
oscillations, which are slightly stronger with FCCWI. 
It is important to emphasize that, despite the source covariance 
being unknown during the inversion, the wave speed reconstruction 
obtained using cross-correlation is not significantly degraded. 
We now proceed to reconstruct the source covariance in the 
cross-correlation framework. In \cref{figure:toy-reconstruction-scov}, 
we compare the reconstructions obtained using different wave speed 
models: first, a homogeneous model corresponding to the constant 
background, and second, the reconstructed model obtained in 
\cref{figure:toy-reconstruction-cp_cc}.

% =============================
\begin{figure}[ht!] \centering
  \subfloat[Reconstruction using homogeneous background wave speed.]
  {\makebox[15em][c]{\includegraphics[scale=1]{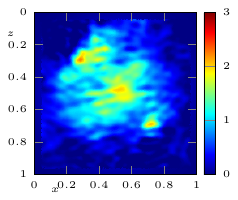}}}
  \hspace*{3em}
  \subfloat[Reconstruction using the reconstructed
            wave speed in \cref{figure:toy-reconstruction-cp_cc}.]
  {\makebox[15em][c]{\includegraphics[scale=1]{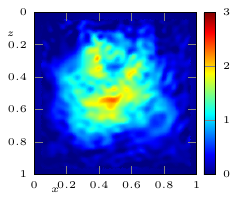}}}
  \caption{Source covariance reconstruction depending on the 
           (fixed) wave speed models used during the iterations.}
  \label{figure:toy-reconstruction-scov}
\end{figure}

We observe that the reconstruction of the source covariance 
does not perform as well as that of the wave speed. While the 
overall shape of the patterns is approximately recovered, 
the magnitude remains lower than the true values shown in 
\cref{figure:toy-true}. The wave speed model used in the 
source covariance reconstruction has an impact on 
the results: using the wave speed obtained from the inversion 
in \cref{figure:toy-reconstruction-cp} leads to improved accuracy. 
This suggests that a two-stage procedure should be adopted. 
First, the wave speed is reconstructed, as it is less sensitive 
to the unknown source covariance. Then, in a second stage, the 
source covariance can be recovered using the estimated wave 
speed model.

% -------------------------------
\subsection{Experiment 2: Seismic configuration with partial data}
\label{subsection:ot}
% -------------------------------

We consider a configuration representative of seismic imaging, 
in which only partial data are available at the surface. 
We use the Overthrust (OT) seismic model \cite{Aminzadeh1994seg}, 
which provides both the wave speed and density backgrounds
(see \cref{figure:ot-true}). The computational domain has size 
\num{20}$\times$\num{4.65}\si{\kilo\meter\squared}, and the background 
consists of layers with varying material properties. 
The wave speed varies strongly, approximately by a 
factor of 2.5 between the surface and the deepest 
region (from 2.4\si{\km\per\second} to 6\si{\km\per\second}). 
The synthetic source covariance (see \cref{figure:ot-true_scov}) is constructed 
as a smooth background field with a mild imprint of 
the layer structures.

\begin{figure}[ht!] \centering
  \subfloat[Wave speed.]
  {\includegraphics[scale=1]{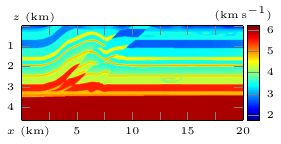}} \hfill
  \subfloat[Density.]
  {\includegraphics[scale=1]{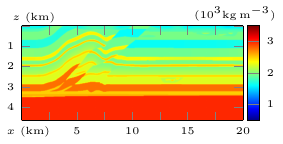}} \hfill
  \subfloat[Source covariance.]
  {\includegraphics[scale=1]{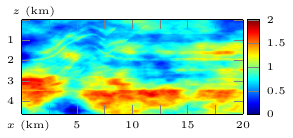}
   \label{figure:ot-true_scov}}
  \caption{The seismic model OT used for inversion. The domain
           has size \num{20}$\times$\num{4.65}\si{\kilo\meter\squared}.
           The data for inversion are only acquired 
           near the surface.}
  \label{figure:ot-true}
\end{figure}

For the numerical configuration, a free-surface boundary 
condition (zero pressure) is imposed at $z=0$, while absorbing 
boundary conditions are applied on the other sides of the domain. 
The synthetic data are generated using acquisition devices located 
near the surface only: we consider \num{399} receivers for 
$\rcvs$ in \cref{equation:misfits}, 
uniformly distributed every 50\si{\meter} at a depth of \num{50}\si{\meter}. 
Only half of these positions are used for $\srcs$ 
in \cref{equation:misfits}: these correspond to the source
positions in the traditional FWI setting, 
and to one of the cross-correlation positions 
in \cref{algorithm:forward-cc-hdg}. 
The synthetic data are generated with the models of 
\cref{figure:ot-true}, and white noise is added to 
avoid inverse crime.
For the reconstruction, we start from initial models that do not 
contain any information about the layered structure, and instead 
consist of a one-dimensional variation with depth, as shown in 
\cref{figure:ot-start}.

\begin{figure}[ht!] \centering
  \subfloat[Wave speed.]
  {\includegraphics[scale=1]{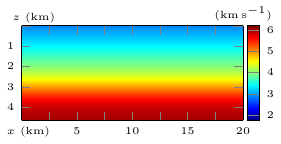}} \hfill
  \subfloat[Density.]
  {\includegraphics[scale=1]{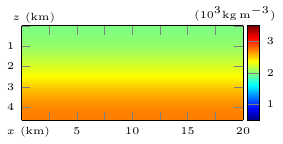}} \hfill
  \subfloat[Source covariance.]
  {\includegraphics[scale=1]{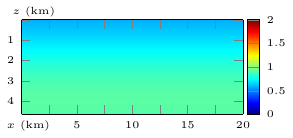}}
  \caption{Initial guesses for the reconstruction of the 
           OT model \cref{figure:ot-true}. They consist
           in one-dimensional variation in depth.}
  \label{figure:ot-start}
\end{figure}

% -----------------------------------
\subsubsection{Wave speed reconstruction}
% -----------------------------------

We begin by investigating the reconstruction 
of the wave speed, while keeping the other 
parameters (density and source covariance) 
fixed at their initial values shown 
in \cref{figure:ot-start}, i.e., they remain
unknown but are not inverted. 
The reconstruction is performed using data in 
the frequency range from 2 to 10\si{\Hz}. 
These frequencies are processed sequentially, 
with 30 minimization iterations carried out 
for each frequency, and the reconstruction 
obtained at a given frequency used to initialize 
the inversion at the next. The resulting 
reconstructions obtained using $\misfitfwi$ 
and $\misfitcc$ are shown in \cref{figure:ot-reconstruction-cp}.

\begin{figure}[ht!] \centering
  \subfloat[FWI reconstruction using $\misfitfwi$.]
  {\makebox[14em][c]{\includegraphics[scale=1]{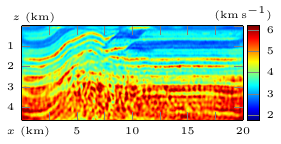}}} \hspace*{1em}
  \subfloat[FCCWI reconstruction using $\misfitcc$.]
  {\makebox[22em][c]{\includegraphics[scale=1]{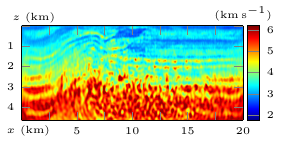}}
  \label{figure:ot-reconstruction-cp_cc}}
\caption{Wave speed reconstructions of the OT model 
         using either the direct wavefield data (FWI) 
         or using cross-correlation data (FCCWI). 
         The initial models for inversion are shown 
         in \cref{figure:ot-start}. The density and 
         source covariance are not inverted and are 
         kept fixed at their initial guess values.}
  \label{figure:ot-reconstruction-cp}
\end{figure}

The FWI reconstruction, which relies on data from
controlled sources, provides a better 
resolution than the cross-correlation-based inversion
with FCCWI. 
In the later case, the unknown source
covariance increases the complexity of the
inversion and can lead to a loss of accuracy.
Nevertheless, the different layers of the OT
model, along with the corresponding material
properties, are reasonably well recovered in the
reconstruction. With FWI, these layers are more
clearly resolved, particularly near the surface,
while in both approaches the deepest structures
remain more difficult to reconstruct.
This experiment highlights that inversion based on
cross-correlation in seismic imaging is inherently
more difficult due to the unknown source covariance,
whereas FWI alleviates this uncertainty, albeit at
the cost of a more demanding acquisition campaign.

% -----------------------------------------------
\subsubsection{Source covariance reconstruction}
% -----------------------------------------------

We now reconstruct the source covariance in the
cross-correlation inversion. The wave speed and 
density are kept fixed, and we compare in 
\cref{figure:ot-reconstruction-scov} using the 
starting wave speed (\cref{figure:ot-start}) or 
the reconstruction obtained in the previous stage,
see \cref{figure:ot-reconstruction-cp_cc}.
The frequency range for the data sequentially varies 
from 2 to 10\si{\Hz} and we perform 30 iterations per 
frequency.

\begin{figure}[ht!] \centering
  \subfloat[Reconstruction using the starting 
            wave speed in \cref{figure:ot-start}.]
  {\makebox[16em][c]{\includegraphics[scale=1]{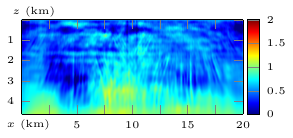}}} \hspace*{2em}
  \subfloat[Reconstruction using the reconstructed
            wave speed in \cref{figure:ot-reconstruction-cp_cc}.]
  {\makebox[16em][c]{\includegraphics[scale=1]{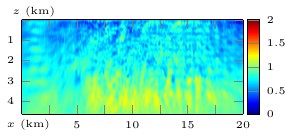}}}
\caption{Source covariance reconstruction for 
         the OT model starting from the initial 
         models shown in \cref{figure:ot-start}.}
  \label{figure:ot-reconstruction-scov}
\end{figure}

We observe that both reconstructions fail to adequately
capture the main features of the source covariance.
Using the initial wave speed further leads to a
low-amplitude reconstruction of the source covariance
across the domain. The reconstruction obtained using
the wave speed recovered from the first-stage inversion
yields a slightly improved magnitude, but the underlying
structures remain unresolved.
This test case highlights the difficulty of recovering the
source covariance in the seismic setting using an $l2$-norm
for the misfit functional, as defined in \cref{equation:misfits}.
This can be expected, since the $l2$-norm is known to be an 
indicator of phase discrepancy but does not necessarily 
capture amplitude (\cite{Virieux2009,faucher2017}).
Consequently, this misfit is well-suited for reconstructing the
wave speed, which primarily affects the phase and can be recovered
relatively accurately in \cref{figure:ot-reconstruction-cp}, 
but it is less effective for the source covariance,
which predominantly influences the amplitude. 
For the same reason, reconstructing the density is more 
difficult. 
% in this context, cf.~\cite{Virieux2009,faucher2017} and the references therein. 
To improve the reconstruction of the source covariance,
alternative misfit functionals need to be explored, this 
is part of ongoing work.

% -------------------------------
\subsection{Experiment 3: 3D test-case with high-contrast inclusions}
\label{subsection:3d}
% -------------------------------

We consider a three-dimensional case to assess the feasibility 
of the proposed approach in a more realistic setting and to 
evaluate its computational cost. 
The domain has dimensions \num{2.5}$\times$\num{1.5}$\times$\num{1}\si{\kilo\meter\cubed}, and we consider a free-surface condition
at the surface ($z=0$) and absorbing boundary conditions elsewhere.
The wave speed model contains several high-contrast inclusions and
is shown \cref{figure:3d:cp-true-start}. 
It ranges from approximately \num{1.5}\si{\km\per\second} near surface 
to \num{4.5}\si{\km\per\second} within the inclusions. 
For the source covariance profile, we consider a smoothing of the 
wave speed, see \cref{figure:3d:scov-true}, while the density is 
kept constant and equal to $1$ throughout this experiment.
Synthetic data are generated using receivers located near the 
surface at a depth of $50$\si{\meter}. The receivers are arranged 
on a structured grid with a spacing of $50$\si{\meter}. For the 
active-source configuration, as well as for the set $\srcs$ used 
in the cross-correlation inversion, every second receiver in each 
direction is selected as a source.

\begin{figure}[ht!] \centering
  \subfloat[True wave speed model.] 
           {\includegraphics[scale=0.55]{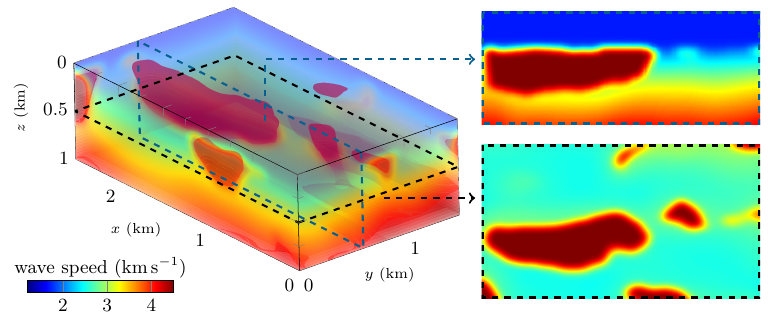}}
  \subfloat[True source covariance model.]
           {\includegraphics[scale=0.55]{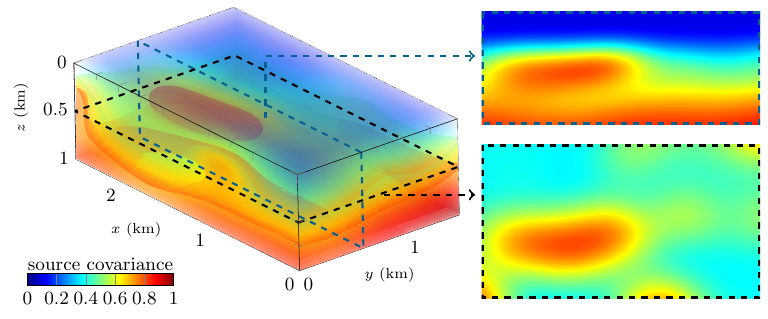}
            \label{figure:3d:scov-true}}

  \subfloat[Starting wave speed model for inversion, the starting 
            source covariance has similar pattern but rescaled between 0 and 1.]
  {\makebox[28em][c]{\includegraphics[scale=0.55]{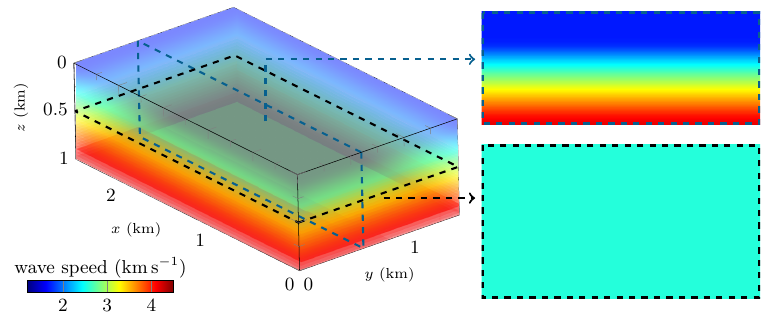}}
   \label{figure:3d:cp-start}}

  \caption{Wave speed and source covariance for the 3D test case 
           of size \num{2.5}$\times$\num{1.5}$\times$\num{1}\si{\kilo\meter\cubed}.
           The density is taken fixed to $\rho=1$. For inversion, the pressure 
           wavefield data are only acquired near surface.}
  \label{figure:3d:cp-true-start}
\end{figure}

The reconstruction is carried out with frequencies from $2$ to $7$\si{\Hz}. 
The initial wave speed model is pictured \cref{figure:3d:cp-start}.
We do not assume any of the inclusion, and start from a one-dimensional 
variation in depth only. This test-case is particularly challenging due 
to the computational cost associated with 3D inversion and the difficulty 
of reconstructing high-contrast structures (initially unknown) from partial 
surface measurements.
The wave speed reconstructions obtained using $\misfitfwi$ (FWI)
and $\misfitcc$ (FCCWI) are shown in 
\cref{figure:3d:reconstruction}.
The reconstruction is performed using frequencies from $2$ to $7$\si{\Hz}. 
The initial wave speed model is shown in \cref{figure:3d:cp-start}. 
No prior information about the inclusions is assumed; instead, the inversion 
starts from a one-dimensional model that varies only with depth.
This test case is particularly challenging due to both the computational 
cost of three-dimensional inversion and the difficulty of reconstructing 
initially unknown high-contrast structures from partial measurements 
acquired at the surface.
The wave-speed reconstructions obtained with $\misfitfwi$ (FWI) 
and $\misfitcc$ (FCCWI) are shown 
in \cref{figure:3d:reconstruction}. 
In the latter case, the source covariance is not inverted and is 
kept fixed at its initial value during the procedure.

\begin{figure}[ht!] \centering
  \subfloat[FWI reconstruction using $\misfitfwi$.]
  {\includegraphics[scale=0.55]{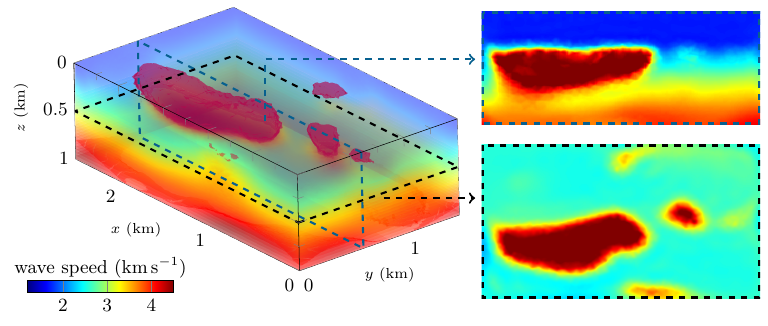}} \hspace*{0.50em}
  \subfloat[FCCWI reconstruction using $\misfitcc$.]
  {\includegraphics[scale=0.55]{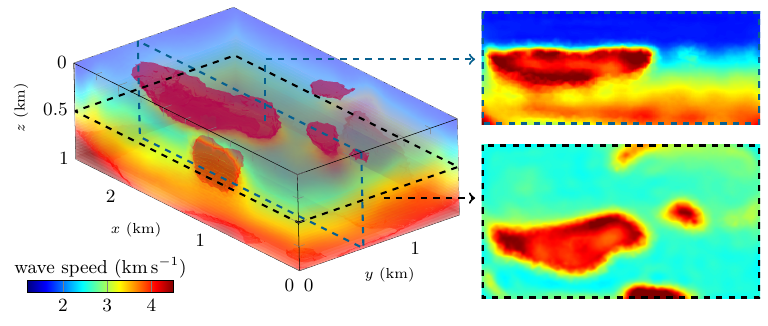}}
  \caption{3D wave speed reconstructions using either the 
           direct wavefield data (FWI) or using cross-correlation
           data (FCCWI). 
           The true and initial models for inversion are shown 
           in \cref{figure:3d:cp-true-start}.}
  \label{figure:3d:reconstruction}
\end{figure}

We observe that the subsurface inclusions are well reconstructed
in both cases. The geometry of the largest inclusion is accurately
recovered, and the smaller inclusions are also clearly visible.
Notably, the inclusions located close to the boundaries,
particularly the one near $y=0$ in the bottom-right slices of
\cref{figure:3d:reconstruction}, are better resolved with
cross-correlation data than with the direct wavefield. On the
other hand, the wave-speed amplitude within the inclusions is
more accurately recovered by FWI, whereas 
FCCWI slightly underestimates it, this is possibly due to 
the lack of information regarding the source covariance. 
Nevertheless, both approaches yield satisfactory results and 
successfully recover high-contrast inclusions.

\paragraph{Computational cost}
This experiment is carried out on the cluster
PlaFRIM\footnote{Plateforme Fédérative pour la Recherche en Informatique et Mathématiques, \url{https://www.plafrim.fr/}.}
using 48 processors.
The reconstruction employs integer frequencies from 2 to 7~\si{\Hz},
with 30 iterations performed at each frequency, resulting in a total
of 180 iterations. The numerical implementation based on the HDG
method adapts the size of the computational system to the frequency
through $p$-adaptivity, as described in \cite{Faucher2020}.
For this test case, FWI requires approximately 1 minute per iteration,
whereas FCCWI requires approximately 4 minutes
per iteration. This significant increase in computational cost is mainly
due to the larger number of linear systems that must be solved, together
with the additional (full) right-hand sides that need to be assembled and processed.

% --------------------------------------------------------------------------
\section{Conclusion}
\label{section:conclusion}

In this work, we presented a quantitative reconstruction methodology 
for passive imaging based on the expected value of cross-correlations 
of ambient wavefields. Our approach exploits the complete cross-correlated 
wavefield rather than relying solely on travel times or first arrivals.
The expected cross-correlation between two signals is expressed in terms 
of deterministic Green's functions, with the source covariance introduced 
as an additional unknown in the inversion. 
We have formulated the subsequent nonlinear inverse problem 
as a minimization procedure, in the spirit of full waveform 
inversion, however adapted to cross-correlation data. 
To efficiently compute the gradient of the misfit functional, we derived 
the adjoint-state formulation involving two adjoint problems. We further 
detailed the numerical implementation based on the Hybridizable 
Discontinuous Galerkin method, adapting to its first-order formulation 
and static condensation.
The numerical experiments demonstrate the feasibility of the proposed approach. 
Although controlled-source imaging remains more accurate when available, the 
results indicate that passive imaging using the full cross-correlation waveform 
can recover meaningful quantitative information about the medium. 
In contrast, the reconstruction of the source covariance proved more difficult, 
and investigating dedicated linear inversion strategies is part of future research.
We will also consider the correlations between different fields that will lead to 
a matrix-valued source covariance.
Future work will also focus on applying the methodology to realistic applications 
such as helioseismology and seismic Earth monitoring, in order to assess the validity 
of the underlying assumptions and evaluate the performance of the proposed framework 
under realistic configurations.

\section*{Acknowledgments}
  The authors acknowledge funding by the European Union 
  with ERC Project \textsc{Incorwave} -- grant 101116288.
  Views and opinions expressed are however those of the authors 
  Views and opinions expressed are however those of the authors 
  only and do not necessarily reflect those 
  of the European Union or the European Research Council 
  Executive Agency (ERCEA). Neither the European Union nor the
  granting authority can be held responsible for them.

% --------------------------------------

% \appendix
% \input{sections/appendix_attraction}

% --------------------------------------
\bibliographystyle{siam}
\bibliography{bibliography}

% --------------------------------------
\end{document}